\documentclass[12pt]{amsart}
\usepackage{amsmath,amssymb,latexsym, amsthm, amscd, mathrsfs, stmaryrd,  color} 
\usepackage[all]{xy}

\theoremstyle{definition}
\theoremstyle{plain}

\newcommand{\A}{\mathfrak o}

\newcommand{\bj}{\mathbf{j}}

\newcommand{\bU}{\mathbf{U}}

\newcommand{\cA}{\mathcal{A}}

\newcommand{\C}{\mathfrak{c}}

\newcommand{\co}{\textup{co}}

\newcommand{\f}{\mathbf f}

\newcommand{\fa}{\mathfrak{a}}

\newcommand{\fc}{\mathfrak{c}}

\newcommand{\Hom}{\textup{Hom}}

\newcommand{\KK}{\mbf{K}}

\newcommand{\mbb}{\mathbb}
\newcommand{\mbf}{\mathbf}

\newcommand{\mm}{\eta}
\newcommand{\MX}{\Xi}  
\newcommand{\mrm}{\mathrm}

\newcommand{\nn}{\mathfrak{n}}

\newcommand{\Q}{\mbb Q}

\newcommand{\redtext}[1]{\textcolor{red}{#1}}
\newcommand{\ro}{\mrm{ro}}

\newcommand{\Sj}{{\mbf S}^{\fc}_{n,d}}

\newcommand{\SP}{\mrm{Sp}}

\newcommand{\ve}{\varepsilon}

\newcommand{\X}{\mathcal X}

\newcommand{\ZZ}{\mathbb{Z}}

\newcommand{\red}[1]{{\color{red}#1}}

\usepackage{enumitem }

\theoremstyle{definition}
\newtheorem{Def}{Definition}[section] 

\newtheorem{rem}[Def]{Remark}

\theoremstyle{plain}
\newtheorem{prop}[Def]{Proposition}
\newtheorem{thm}[Def]{Theorem}
\newtheorem{lem}[Def]{Lemma}

\newtheorem{cor}[Def]{Corollary}

\newtheorem{thrm}{Theorem}

\newcommand{\ji}{\jmath \imath}

\newcommand{\Sji}{\mbf S^{\ji}_{\nn, d}}
\newcommand{\eji}{\check{\mbf e}}
\newcommand{\fji}{\check{\mbf f}}
\newcommand{\kji}{\check{\mbf k}}

\newcommand{\tji}{\check{\mbf t}}

\newcommand{\Kcji}{\dot{\mbf K}^{\ji}_{\nn}}

\newcommand{\Kc}{\dot{\mbf K}^{\fc}}

\renewcommand{\ij}{\imath \jmath}

\newcommand{\Sij}{\mbf S^{\ij}_{\nn, d}}
\newcommand{\eij}{\hat{\mbf e}}
\newcommand{\fij}{\hat{\mbf f}}
\newcommand{\kij}{\hat{\mbf k}}

\newcommand{\tij}{\hat{\mbf t}}

\newcommand{\Kcij}{\dot{\mbf K}^{\ij}_{\nn}}

\newcommand{\ii}{\imath \imath}

\newcommand{\Sii}{\mbf S^{\ii}_{\eta, d}}

\newcommand{\Kcii}{\dot{\mbf K}^{\ii}_{\mm}}

\newcommand{\jjw}{\mbox{$\jmath$\kern-1.3pt$\jmath$}}
\newcommand{\jiw}{\mbox{$\jmath$\kern-0.8pt$\imath$}}
\newcommand{\ijw}{\mbox{$\imath$\kern-1.8pt$\jmath$}}
\newcommand{\iiw}{\mbox{$\imath$\kern-1.0pt$\imath$}}

\title[Multiplication formula]{Affine flag varieties and quantum symmetric pairs, II. Multiplication formula}    
\makeindex

\author[Z. Fan and Y. Li]{Zhaobing Fan and Yiqiang Li}
\address{School of science, Harbin Engineering University, Harbin, China 150001}
    \email{fanz@math.ksu.edu  (Fan)}
\address{Department of Mathematics, University at Buffalo, The State University of New York}
    \email{yiqiang@buffalo.edu (Li)}

\date{\today}

\begin{document}

\begin{abstract}
We establish a multiplication formula for a tridiagonal standard basis element  in the idempotent version, i.e., the Lusztig form,  of the coideal subalgebras of
quantum affine $\mathfrak{gl}_n$ arising from the geometry of affine partial flag varieties of type $C$.
We apply this formula to obtain the stabilization algebras $\Kc_n$, $\Kcji$, $\Kcij$ and $\Kcii$,
which are idempotented coideal subalgebras of quantum affine $\mathfrak{gl}_n$.
The symmetry in the formula  leads to an isomorphism of the idempotented coideal subalgebras $\Kcji$ and $\Kcij$ with compatible monomial, standard and canonical bases.
\end{abstract}

\maketitle



\section*{Introduction}

In this article, we continue our study in~\cite{FLLLWa}, joint with Chun-Ju Lai, Li Luo and Weiqiang Wang,  of the Schur algebras  and their stabilization algebras  of affine $n$-step partial flag varieties of type $C$.
As the main result in this paper, we establish  a multiplication formula for a tridiagonal standard basis element in these algebras.
This result was supposed to be  included in {\it loc. cit.}
as a critical first step, but was extracted from the {\it loc. cit.} partly due to the length of the paper.
Instead, a more conceptual multiplication formula of a new generator, denoted by $\mbf f_A$,  was used as a substitute in {\it loc. cit.}
A third multiplication formula is further secured in~\cite{FLLLWb} via a Hecke-algebra approach.
All three formulas look drastically different, have their own advantage, and coexist coherently.
They reflect the richness of (the structure of) these algebras, even though the establishment of each one of them is a painstaking task.

 The multiplication formula in this article possesses a remarkable symmetry and, as an application, it
 yields an isomorphism of two seemingly-different classes of Schur algebras in ~\cite{FLLLWa} with compatible monomial, standard and canonical bases.
 The proof of the latter isomorphism result  is the only one available as far as we know.
 As a second application,  the multiplication formula in the setting of affine Grassmannian in type $C$ should lead to an explicit formula for parabolic affine Kazhdan-Lusztig polynomials  similar to~\cite[Section 7]{LW14}.
We expect our formula to play a role in categorifications of these algebras as well.

In what follows, we discuss in more details the background and results of this article.

\subsection{Overview}
It is well-known that the convolution algebras of flag varieties provide a natural geometric model for Hecke algebras.
As an important feature from this model,  the Kazhdan-Lusztig bases ~\cite{KL79} of the Hecke algebras can be interpreted as
intersection cohomology complexes therein and from which one can deduce nontrivial properties such as positivity.

Later, Beilinson, Lusztig and MacPherson  ~\cite{BLM90}
observed that partial flag varieties  of type $A$ are a geometric setting for quantum
$\mathfrak{gl}_n$.
For other classical partial flag varieties, it is only known recently that
they are governed by coideal subalgebras of quantum $\mathfrak{gl}_n$ (\cite{BKLW14, BLW14}).

There is an affinization of the above results making a connection between  affine flag varieties of type $A$ and $C$ and
quantum affine $\mathfrak{gl}_n/\mathfrak{sl}_n$ ~\cite{GV93,Lu99, Lu00,  SV00, Mc12}
and their coideal subalgebras (\cite{FLLLWa}).

More precisely, we provide a description of
the convolution algebra of affine partial flag varieties of type $C$  in the  work ~\cite{FLLLWa}.
Among others, we show that it is controlled by certain coideal subalgebra of quantum affine $\mathfrak{gl}_n$
and in turn provides canonical bases  for the latter algebras.
As a crucial ingredient, we show that the convolution algebra
admits a set of  bar-invariant multiplicative generators, denoted by $\f_A$ therein,  parametrized by
certain tridiagonal matrices.  Roughly speaking, the bar-invariant basis $\f_A$  corresponds to  a product of Chevalley generators
in higher-rank convolution algebras in a specific order.
In particular, one can derive  a multiplication formula for $\f_A$ by repetitively applying
the multiplication formulas for Chevalley generators, which turns out to be simple.
This leads to a construction of
the idempotented version of the associated coideal subalgebras  of quantum affine $\mathfrak{gl}_n$ and their canonical bases
after a suitable stabilization following ~\cite{BLM90} and ~\cite{FL14}.

There is yet another natural set of multiplicative generators consisting of the tridiagonal
standard basis elements $[A]$,
which are parametrized by the same set of tridiagonal matrices for the generators $\f_A$.
The transition matrix of the two bases turns out to be unitriangular.
The purpose of this paper is to establish a multiplication formula for the tridiagonal  standard basis elements $[A]$,
which then provides a direct construction of the affine coideal subalgebras  of quantum affine $\mathfrak{gl}_n$ and their canonical bases in ~\cite{FLLLWa}
with some further new and interesting symmetries.

\subsection{Main results}

Since we know essentially the multiplication formula for $\f_A$,
the complexity to establish  the multiplication formula for a tridiagonal basis element $[A]$
is incorporated in the transition matrix of the two bases,
which ends up producing coefficients involving  terms $(v^m-1)$ for various $m$.
The latter terms vanish in finite types or affine type $A$ since $\f_A$ and $[A]$ coincide for the necessary matrices in these cases and
hence the transition matrix is the identity matrix.
A rough form of the multiplication formulas is as follows, and we refer to \eqref{eq-Eij}, \eqref{Mdn}, \eqref{Theta:n},
\eqref{n(S, T)}, \eqref{AST2} and \eqref{hST2} for unexplained notations.

\begin{thrm}  [Theorem~\ref{mult-stand-basis2}]
\label{mult-formula:Intr}
Let $\alpha = (\alpha_i)_{i\in \mbb Z} \in \mbb N^{\mbb Z}$ such that $\alpha_i = \alpha_{i+n}$ for all $i\in \mbb Z$.
If $A, B \in \MX_{n,d}$  satisfy that $\co(B) = \ro(A)$ and $B - \sum_{1\leq i \leq n} \alpha_i E_{\theta}^{i, i+1}$ is diagonal,
then we have
\begin{equation*}
[B] * [A] = \sum_{S, T} v^{h_{S,T}} n(S, T)
\begin{bmatrix}
A_{S,T}\\ S
\end{bmatrix}_{\mathfrak b} [A_{S,T}],
\end{equation*}
where the sum runs over all $S, T \in \Theta_{n}$ subject to Condition \eqref{star},
$\ro(S) =\alpha$ and $\ro(T) = \alpha^J$,
$A-T+\check{T} \in \Theta_n$, and $A_{S,T} \in\MX_{n,d}$.
\end{thrm}

In many respects, the way we derive this formula has many similarities  to, and is partially based on,
the one on quantum affine $\mathfrak{gl}_n$ constructed in ~\cite{FL15}.
Note that the latter formula is first discovered in ~\cite{DF13} via a Hecke-algebraic approach.
Besides the new construction in ~\cite{FL15}, the identification of the affine type-$A$ analogue of $\f_A$ and $[A]$ for $A$ bidiagonal indeed gives
an easier way to deduce this formula.

Even though the multiplication formula is rather involved,
we are able to obtain sufficient conditions on when a leading term with coefficient $1$ is produced under such multiplication,
which allow us to observe a BLM-type stability property as $d$ goes to $\infty$.
The stabilization also allows us to formulate a limit algebra $\Kc_n$ for the family of the convolution algebras $\{\Sj\}_{d\ge 1}$.
The algorithm for the construction of a monomial basis for $\Sj$ leads to an algorithm for a monomial basis for $\Kc_n$.

\begin{thrm}  [Theorem~\ref{thm:CBgl}]
\label{CBgl:Intr}
The algebra $\Kc_n$ admits a monomial basis and a canonical basis.
\end{thrm}

One can define another algebra $\KK^{\fc}_n$ so that $\Kc_n$ is  identified with
the idempotented version of  a coideal subalgebra of quantum affine $\mathfrak{gl}_n$ as is shown in ~\cite{FLLLWa}.

Similarly, the other families of convolution algebras $\{{\mbf S}^{\ji}_{\nn,d} \}_d, \{{\mbf S}^{\ij}_{\nn,d} \}_d,$
and $\{{\mbf S}^{\ii}_{\mm,d} \}_d$
admit similar stabilizations which lead to limit algebras $\Kcji, \Kcij$ and $ \Kcii$, respectively.
We also establish the counterparts of Theorem~\ref{CBgl:Intr} for the algebras $\Kcji, \Kcij$ and $\Kcii$, which require
some additional work (it follows a strategy similar to \cite{FL14}).
Despite quite complicated, the structure constants in
the multiplication formula in Theorem ~\ref{mult-formula:Intr} manifest remarkable symmetries,
reflecting the shift by half-period in parametrization matrices,
which allow us to show that

\begin{thrm}[Theorem~\ref{thm:tauK}]
There is an isomorphism $\Kcji \cong \Kcij$ with compatible monomial, standard and canonical bases.
\end{thrm}

Simultaneously, a Hecke-algebraic approach has been developed in a companion paper \cite{FLLLWb}, which
reproduces most of the main results of  this paper in different forms.
In light of a result in {\it loc. cit.}
the algebras denoted by  the same notations $\Kc_n$, $\Kcji$, $\Kcij$ and $\Kcii$
in the paper and ~\cite{FLLLWa, FLLLWb} are isomorphic and they are idempotent versions
of coideal subalgebras of quantum affine $\mathfrak{gl}_n$.
Having the three approaches available indicates the rich structures in these algebras arising from geometry.

\subsection{The organization}

In Section ~\ref{rec}, we recall the setting from ~\cite{FLLLWa} on the convolution algebras $\Sj$, $\Sji$, $\Sij$ and $\Sii$.

In Section ~\ref{chap:multi},
we obtain a multiplication formula for a tridiagonal standard basis element in
the convolution algebra $\Sj$. The proof is rather involved  taking up the whole long section, and
the formula
is the key to the remaining sections on the structures of $\Sj$ which leads to the construction of the limit algebra $\Kc_n$.

In Section~\ref{chap:Kjj},
we obtain a monomial basis  for the convolution algebra $\Sj$
based on the multiplication formula obtained in Section ~\ref{chap:multi}. In particular,
it follows that the standard basis elements parametrized by tridiagonal matrices form a generating set for the algebra $\Sj$.
This multiplication formula admits a remarkable stabilization property which allows us
to construct a limit algebra $\Kc_n$ for the family of convolution algebras $\{\Sj\}_d$.
We then construct a monomial basis and canonical basis for $\Kc_n$, as well as a
surjective homomorphism from $\Kc_n$ to $\Sj$.

In Section~\ref{chap:K234},
we adapt the multiplication formula from Section ~\ref{chap:multi} for
the other variants of convolution algebras:  ${\mbf S}^{\ji}_{\nn,d}$, ${\mbf S}^{\ij}_{\nn,d}$, ${\mbf S}^{\ii}_{\mm,d}$.
This then allows us to construct the corresponding limit algebras $\Kcij$, $\Kcij$, and $\Kcii$, respectively.
Monomial bases and canonical bases for these algebras are constructed.

\subsection{Acknowledgement}
The paper  is grown out from the Project~\cite{FLLLWa, FLLLWb} and thus we thank our collaborators for fruitful collaborations and
their generosity in sharing their ideas and allowing us to publish this paper separately.
We also thank our collaborators Chun-Ju Lai and Li Luo  for cross-checking our multiplication formula and are grateful for Weiqiang Wang for  his leadership and tuning up an earlier version of this article.
Z. Fan is partially supported by the NSF of China grant 11671108, the NSF of Heilongjiang
Province grant LC2017001 and the Fundamental Research Funds for the central universities
GK2110260131.
Y. Li is partly supported by the National Science Foundation under the grant DMS 1801915.

\tableofcontents




\section{Recollections from ~\cite{FLLLWa}}
\label{rec}

In this section, we recall the setting from ~\cite{FLLLWa} and fix some notations.
There is no new result in this section.

\subsection{The space of affine flags of type $C$ as symplectic lattice chains}
Let $\mbb N =\{0, 1, 2, \ldots\}$ and $\mbb Z=\{ 0, \pm 1, \pm 2, \ldots\}$.
For $a \in \mathbb{Z}$ and $b \in \mathbb N$, we define
\begin{equation}
  \label{eq:binomial}
\begin{bmatrix}
a\\
b
\end{bmatrix}
=\prod_{1\leq i\leq b} \frac{v^{2(a-i+1)}-1}{v^{2i}-1}, \quad [a] = \begin{bmatrix}
a\\
1
\end{bmatrix},
 \quad \text{ and } [a]! = \prod_{1\leq i \leq a} [i].
\end{equation}

Let $k$ be a finite field of $q$ elements.
Let $F= k((\ve))$ be the field of formal Laurent series over $k$ and $\A=k[[\ve]]$
the ring of formal power series.
Let $V$ be an $F$-vector space. A lattice $\mathcal L$ in $V$ is a free $\A$-module such that $\mathcal L\otimes_{\A} F=V$.

Assume  further that $V$ is equipped with a non-degenerate symplectic form $(-, -)$. Hence $V$ is even dimensional, say $2d$.
Let $\SP(V)$ be the group of isometries with respect to the form $(-, -)$.
For any lattice $\mathcal L\in V$, we set $\mathcal L^{\#}=\{ v\in V | (v, \mathcal L)\subseteq \A\}$.
Then the space $\mathcal L^{\#}$ is a lattice of $V$ such that $(\mathcal L^{\#})^{\#}=\mathcal L$.
The $\#$ operation enjoys the following properties that we shall use freely later on.
For any two lattices $\mathcal L, \mathcal L'$ of $V$,
$(\mathcal L+\mathcal L')^{\#} = \mathcal L^{\#} \cap \mathcal L'^{\#}$
and $(\mathcal L \cap \mathcal L')^{\#} = \mathcal L^{\#} + \mathcal L'^{\#}$.
We are interested in the symplectic lattices, which are those homothetic to a lattice $\Lambda$ such that
$\ve \Lambda \subseteq \Lambda^{\#} \subseteq \Lambda$.

Once and for all, we fix an even number
\[
n = 2r +2, \quad \mbox{for some} \ r \in \mbb N.
\]
Let $\X^{\C}_{n, d}$ be the set of all chains $L=(L_i)_{i\in \mbb Z}$ of symplectic lattices in $V$ subject to the following conditions.
\[
L_z \subseteq L_{z+1}, \quad L_z = \ve L_{z+n}, \quad L_z^{\#} = L_{-z-1}, \quad \forall  z\in \mbb Z.
\]
The group $\SP(V)$ acts naturally on $\X^{\C}_{n, d}$.
Let
$$
\Lambda^{\C}_{n, d} = \{ \lambda =(\lambda_i)_{i\in \mbb Z} \in \mbb N^{\mbb Z} | \lambda_i= \lambda_{i+n}=\lambda_{-i-1}, \sum_{1\leq i\leq n} \lambda_i = 2d, \lambda_0, \lambda_{r+1} \in 2\mbb N\}.
$$
For each $\lambda \in \Lambda^{\C}_{n, d}$, we set
\begin{align}
\label{L-L}
\X^{\C}_{n, d} (\lambda) = \{ L\in \X^{\C}_{n,d} | |L_i/L_{i-1}| = \lambda_i \quad \forall i\in \mbb Z\},
\end{align}
where $|L_i/L_{i-1}|$ is the dimension of $L_i/L_{i-1}$ as a $k$-vector space.
Then we have
\begin{align}
\X^{\C}_{n, d} = \sqcup_{\lambda \in \Lambda^{\C}_{n, d}} \X^{\C}_{n, d} (\lambda),
\end{align}
as the union of $\SP(V)$-orbits in $\X^{\C}_{n, d}$.
From the analysis of ~\cite{Sa99}, ~\cite{H99} (see also ~\cite[Section 3.2]{FLLLWa}), we see that $\X^{\C}_{n, d}(\lambda)$
can be naturally identified with the homogeneous space
$\SP(V)/P$ where $P$ is a certain parahoric subgroup.
Hence $\X^{\C}_{n,d}$ is a local model of the ind-variety of affine partial flags of type $C$ over $k$.

\subsection{Parametrization matrices}
\label{Para}

The $\SP(V)$-action  on $\X^{\C}_{n, d}$ extends  diagonally on the product $\X^{\C}_{n, d} \times \X^{\C}_{n, d}$.
We recall a parametrization of $\SP(V)$-orbits in $\X^{\C}_{n, d} \times \X^{\C}_{n, d}$.
Let
$\Theta_{n, d}$
be the set of all matrices $A=(a_{ij})_{i, j\in \mbb Z}$ with entries in $\mbb N$ such that
\[
a_{ij} = a_{i+n, j+n}, \quad \forall i, j \in \mbb Z; \quad
\sum_{i=i_0}^{i_0+n-1} \sum_{j\in \mbb Z} a_{ij} = d, \quad \forall i_0 \in \mbb Z.
\]
To a matrix $A\in \Theta_{n, d}$, we associate its row/column vector
$\ro(A)=(\ro(A)_i)_{i\in \mbb Z}$ and $\co(A)=(\co(A)_i)_{i\in \mbb Z}$  by
\[
\ro(A)_i = \sum_{j\in \mbb Z} a_{ij},\quad
\co(A)_j = \sum_{i\in \mbb Z} a_{ij}, \quad \forall i, j \in \mbb Z.
\]
Let
\begin{align}
\label{Xidn}
^{\C}\Xi_{n,d} =\{ A\in \Theta_{n, 2d} | a_{ij}= a_{-i, -j} = a_{n+i, n+j}, \forall i, j \in \mbb Z, a_{00}, a_{r+1, r+1} \in 2\mbb N\}.
\end{align}
We set $E^{ij} $ to be the matrix whose entry at $(k, l)$ is
$1$ if $(k, l) \equiv (i, j) $ mod $n$ and zero otherwise.
We also set
\begin{equation}\label{eq-Eij}
  E^{ij}_{\theta} = E^{ij} + E^{-i, -j}.
\end{equation}
For later use, let
\begin{align} \label{Mdn}
\Xi_{n, d} = \{ A + E^{0, 0} + E^{r+1, r+1} | A\in \ ^{\C} \Xi_{n, d} \}.
\end{align}
Clearly, there is a bijection
\begin{align}
\label{bijection}
^{\C}\Xi_{n, d} \cong \Xi_{n,d}.
\end{align}

To a pair $(L, L') \in \X^{\C}_{n, d} \times \X^{\C}_{n, d}$,
we can attach a matrix $A\in \ ^{\C}\Xi_{n, d}$ whose $(i, j)$-th entry is given by
\[
a_{ij} = |L_i \cap L'_j / (L_{i-1}\cap L'_j + L_i \cap L'_{j-1})|, \quad \forall i, j\in \mbb Z.
\]
From ~\cite[Proposition 3.2.2]{FLLLWa}, the correspondence $(L, L') \mapsto A$ yields a bijection
\begin{align}
\SP(V) \backslash \X^{\C}_{n, d} \times \X^{\C}_{n, d} \to \ ^{\C} \Xi_{n, d}.
\end{align}
We shall denote $\mathcal O_A$ for the $\SP(V)$-orbit parametrized by $A$.
Let $e_A$ be the characteristic function of $\mathcal O_A$.
Under the isomorphism (\ref{bijection}), we can also parametrize the $\SP(V)$-orbits in $\X^{\C}_{n, d} \times \X^{\C}_{n, d}$
by the set $\Xi_{n, d}$.
The notation $e_A$ also makes sense for a matrix $A\in \Xi_{n, d}$.

\subsection{The convolution algebra $\Sj$}
\label{rec-sj}
Let $\Sj$ be the vector space over $\Q(v)$ spanned by the elements $e_A$ for $A\in \ ^{\C}\Xi_{n, d}$.
There is  a multiplication on $\Sj$ defined by
\begin{align}
e_A * e_B = \sum_{C \in {}^{\C}\Xi_{n, d}} g^C_{A, B} (v), e_C, \quad g^{C}_{A, B}(v) \in \mbb Z[v, v^{-1}],
\end{align}
where the specialization of the polynomial  $g^{C}_{A, B}(v)$ at $\sqrt{q}$ is given by
\begin{align}
g^{C}_{A, B} (\sqrt{q}) = \# \{ \tilde L \in \X^{\C}_{n, d} | (L, \tilde L) \in \mathcal O_A, (\tilde L, L') \in \mathcal O_B\}
\end{align}
for any fixed pair $(L, L') \in \mathcal O_C$.
It is known that $g^C_{A, B} =0$ for all but finitely many $C$.
The algebra $\Sj$ is the convolution algebra on $\X^{\C}_{n, d}$ and is called the Schur algebra in ~\cite{FLLLWa}.

To each $A \in \ ^{\C}\Xi_{n, d}$, we set
\[
d^{\C}_A =\frac{1}{2}
\left ( \sum_{ \substack{1\leq i \leq n\\ i\geq k, j< l}} a_{ij} a_{kl} + \sum_{i\geq 0 > j} a_{ij} + \sum_{i \geq r+1 > j} a_{ij} \right ),
\]
and
\[
[A] = v^{-d^{\C}_A} e_A.
\]
Clearly, the various elements $[A]$ form a basis for $\Sj$, called the standard basis of $\Sj$.

Define a partial order $\leq_{\text{alg}}$ on $^{\C}\Xi_{n,d}$ by
\[
A\leq_{\text{alg}} B \quad \mbox{if and only if} \quad
\sum_{k\leq i, l \geq j} a_{kl} \leq \sum_{k\leq i, l\geq j} b_{kl}, \quad \forall i<j.
\]
We write $A <_{\text{alg}} B$ if further $A\neq B$.
By ``lower terms (than $[B])$'', we refer to the terms $[A]$ with $A <_{\text{alg}} B$, $\ro(A) = \ro(B)$ and $\co(A) = \co(B)$.

There is a bar operator $\bar \empty $ on $\Sj$ satisfying $\overline v =v^{-1}$ and
$\overline{[A]} = [A] + \mbox{lower terms}$.
To each matrix $A$, there is a unique element $\{A\}$ in $\Sj$
that is bar-invariant and $\{ A\} - [A] \in  \sum_{B <_{\text{alg}} A} v^{-1} \mbb Z[v^{-1}] [B]$.
The $\{A\}$'s clearly form a basis of $\Sj$, called the canonical basis of $\Sj$.

\subsection{Three variants}
\label{Sji}

We now set
\[
\nn = n-1 = 2r+1 \quad \mbox{for} \ r\geq 1.
\]
We consider the subset $\Xi^{\ji}_{\nn, d}$ of $\Xi_{n, d}$  defined by
\begin{align}
\label{Mjid}
\Xi^{\ji}_{\nn, d} = \{ A\in \Xi_{n,d} | a_{r+1, j} = \delta_{r+1, j}, a_{i, r+1} = \delta_{i, r+1}, \quad \forall i, j \in \mbb Z\}.
\end{align}
We define an idempotent in $\Sj$ by
\begin{align}
\label{bjr}
\bj_r = \sum [A],
\end{align}
where the sum runs over all diagonal matrices in $\Xi^{\ji}_{\nn, d}$.
We define the subalgebra $\Sji$ of $\Sj$ by
\begin{align}
\Sji = \bj_r \Sj \bj_r.
\end{align}
It is known from ~\cite[Section 7.1]{FLLLWa} that $\Sji$ inherits from $\Sj$ a basis of characteristic functions $e_A$,
a standard basis $[A]$ and a canonical basis $\{A\}$ parametrized by the set $\Xi^{\ji}_{\nn, d}$.

On the other hand, we consider the subset $\Xi^{\ij}_{\nn, d}$ of $\Xi_{n, d}$ given by
\begin{align}
\label{Mijd}
\Xi^{\ij}_{\nn, d} = \{ A\in \Xi_{n,d} | a_{0, j} = \delta_{0, j}, a_{i, 0} = \delta_{i, 0}, \quad \forall i, j \in \mbb Z\}.
\end{align}
We define $\bj_0$ for $\Xi^{\ij}_{\nn, d}$  in exactly the same way as $\bj_r$ in (\ref{bjr}) and we set
\begin{align}
\Sij = \bj_0 \Sj \bj_0.
\end{align}
It is known again from ~\cite[Section 8.1]{FLLLWa} that $\Sij$ inherits from $\Sj$ a basis of characteristic functions $e_A$,
a standard basis $[A]$ and a canonical basis $\{A\}$ parametrized by the set $\Xi^{\ij}_{\nn, d}$.

Finally, we set $\eta=\nn -1=2r$ and
\begin{align}
\label{Miid}
\Xi^{\ii}_{\eta, d} = \Xi^{\ji}_{\nn, d} \cap \Xi^{\ij}_{\nn, d} .
\end{align}
We define
\begin{align}
\Sii = \Sij\cap \Sji,
\end{align}
and we know from ~\cite[Section 8.4]{FLLLWa} that
$\Sii$ inherits from $\Sj$ a basis of characteristic functions $e_A$,
a standard basis $[A]$ and a canonical basis $\{A\}$ parametrized by the set $\Xi^{\ii}_{\eta, d}$.

\section{Multiplication formula for tridiagonal standard basis elements}
\label{chap:multi}

In this section, we obtain a multiplication formula in the convolution algebra $\Sj$
for a tridiagonal  standard basis element.
The formula
is the key to the remaining sections on the structures of $\Sj$ which leads to the construction of the limit algebra $\Kc$.
The proof of  the formula is rather involved,  taking up the whole section.

Note that the formulation for the multiplication formulas in this section use the index set ${}^{\C}\Xi_{n,d}$ in \eqref{Xidn} for bases of $\Sj$,
as it is most convenient to use ${}^{\C}\Xi_{n,d}$ in its proof.
It will be reformulated in terms of ${\MX}_{n,d}$ in Section~\ref{chap:Kjj}.

\subsection{The formula}

Let $n=2r+2$ with $r\geq 0$.  We denote
\begin{align}
 \label{Theta:n}
\begin{split}
\Theta_{n} &=\bigsqcup_{d\ge 0} \Theta_{n,d}.  
\end{split}
\end{align}
We also denote
\begin{align}
 \label{Theta:n2}
\widetilde{\Theta}_n
= \{A  \in \text{Mat}_{\ZZ\times \ZZ}(\mbb{\ZZ}) \big \vert A+ pI \in \Theta_n \  \text{for some}\ p\in \mbb N\},
\end{align}
Where $I$ is the identity matrix.
We define a (row shift) bijection  $\check \ : \Theta_{n} \rightarrow \Theta_{n}$ by sending $S=(s_{ij})$ to
\[
\check S = (\check s_{ij})_{i, j\in \mbb Z}, \quad \check s_{ij} = s_{i-1, j} \; (\forall i, j \in \mbb Z).
\]
Given  $A, S, T \in \Theta_{n}$, we set
\begin{align}
\label{AST}
A_{S, T} = A + S - T - (\check S - \check T).
\end{align}
If we write $A_{S, T} = (a'_{ij})_{i, j \in \mbb Z}$, then $a'_{ij} = a'_{i+n, j+n}$ for all $i, j\in \mbb Z$.
But $a'_{ij}$ may be negative, so $A_{S, T}$ is not in $\Theta_{n}$ in general.
Consider the following subset $\Xi_n$ of $\Theta_n$:
\[
\Xi_{n} =\{ A=(a_{ij})_{i, j\in \mbb Z}\in \Theta_{n} \big \vert a_{ij} = a_{-i, -j},   \forall i, j\in \mbb Z \}.
\]

\begin{lem}
 \label{lem:star}
Let $A \in \Xi_{n}$ and  $S, T \in \Theta_{n}$ be such that $A_{S, T} \in \Theta_{n}$.
Then we have $A_{S, T} \in \Xi_{n}$ if  and only if
$S, T$ satisfy
\begin{align}
\label{star}
s_{ij} + s_{-i-1, -j} = t_{ij} + t_{-i-1, -j}, \quad \forall i, j \in \mbb Z.
\end{align}
\end{lem}

\begin{proof}
By substituting $i$ with $i-1$ in Condition (\ref{star}), we  have
$s_{i-1, j} - t_{i-1, j} = t_{-i, -j} - s_{-i, -j}.$
Denoting $A_{S, T}  = (a'_{ij})$, we have thus obtained
\begin{align}
\label{A'-2}
a'_{ij} = a_{ij} + (s_{ij} + s_{-i, -j}) - (t_{ij} + t_{-i, -j}).
\end{align}
It now follows that $a'_{ij} = a'_{-i, -j}$, whence $A_{S, T} \in \Xi_{n}$.

The above argument can be reversed to establish the opposite direction.
\end{proof}

Let
$$
\mathcal J
=
\big( \{-r-1\} \times (-r-1, \infty) \big)  \bigsqcup \big( [-r, -1] \times \mbb Z \big) \bigsqcup \big( \{0\} \times (- \infty, 0) \big)
\subseteq \mbb Z\times \mbb Z.
$$
Recall the quantum $v$-binomials from \eqref{eq:binomial}.
For $S=(s_{ij})$ and $A_{S, T}  = (a'_{ij})$, we set
\begin{align}
\label{[A; S]}
 \begin{bmatrix}
A_{S, T}  \\
 S
 \end{bmatrix}_{\C}
 =
\prod_{(i, j)\in \mathcal J}
\begin{bmatrix} a'_{ij} \\ s_{ij} \end{bmatrix}
\begin{bmatrix} a'_{ij} - s_{ij} \\ s_{-i, -j} \end{bmatrix}
\prod_{\substack{i= 0, -r-1\\ 0 \leq k \leq s_{ii}-1 }} \frac{[a'_{ii} - 2k]}{[k +1]}.
\end{align}

Given three sequences $\alpha=(\alpha_i)_{i\in \mbb Z},\  \gamma=(\gamma_i)_{i\in \mbb Z}$ and $\beta=(\beta_i)_{i\in \mbb Z}$
in $\mbb{N}^{\mbb Z}$ such that their entries are all zeros except at finitely many places,
we define
\begin{equation}
\label{nabc}
n(\alpha, \gamma, \beta)
=
\sum_{\sigma =(\sigma_{ij})} v^{2(\sum_{i, j, l: j > l} \sigma_{ij} \alpha_l + \sum_{i} \sigma_{ii} \alpha_i
+ \sum_{i>k, j < l} \sigma_{ij} \sigma_{kl}) } \frac{\prod_{i\in \mbb Z}[\beta_i]^!} {\prod_{i, j\in \mbb Z}[\sigma_{ij}]^!}
\prod_{\substack{j \in \mbb Z \\ 0 \leq l \leq \alpha_j -\sigma_{jj} -1}} (v^{2 \alpha_j} -v^{2l}),
\end{equation}
where the sum runs over all upper triangular matrices $\sigma=(\sigma_{ij}) \in \text{Mat}_{\ZZ\times \ZZ}(\mbb{N})$
such that  $\ro(\sigma) = \beta$ and $\co(\sigma) = \gamma$.


Given a sequence $\mbf a=(a_i)_{i\in \mbb Z}$, we define the sequence
$\mbf a^J$ whose $i$-th entry is $a_{-i}$ for all $i\in \mbb Z$.
We set
\begin{align}
\label{n(S, T)}
n(S, T) = \prod_{0\leq i \leq r} n(S_i, T_i, (S_{-i-1})^J ),
\end{align}
where $S_i$ and $T_i$ are the $i$-th row vectors of $S$ and $T$, respectively, and $n(\alpha, \gamma, \beta)$ is in (\ref{nabc}).

To $\alpha =(\alpha_i)_{i\in \mbb Z} \in \mbb Z^{\mbb Z}$, we set
\begin{equation}
\label{alpha2}
\alpha^{\#} = (\alpha^{\#}_i)_{i\in \mbb Z}, \quad  \text{where } \alpha^{\#}_i = \alpha_{-i-1}, \, \forall i\in \mbb Z.
\end{equation}

Recall the subset ${}^{\C}\Xi_{n,d}$ of $\Xi_{n}$ from \eqref{Xidn}. Given $A \in {}^{\C}\Xi_{n,d}$ and $S, T \in \Theta_{n}$, we denote
\begin{equation}
 \label{xiAST}
\xi^{\C}_{A, S, T} =
\sum_{\substack{-r-1 \leq i \leq r\\ j> l}} (a_{ij}' -s_{ij}) s_{i l}
- \Big(\sum_{\substack{-r \leq i \leq -1 \\ j > l}}
+ \sum_{\substack{ i =  -r -1, 0 \\  2i -l > j > l }}  \Big ) s_{-i, - j} s_{i l}
- \sum_{\substack{i=-r-1, 0 \\ j< i }} \frac{s_{i j} (s_{i j} -1)}{2}.
\end{equation}
We can now state the general multiplication formulas for the convolution algebra $\Sj$;
for notations $A_{S, T}$, $\begin{bmatrix}
A_{S, T}  \\
 S
 \end{bmatrix}_{\C}$ and $n(S, T)$  see \eqref{AST}, \eqref{[A; S]} and \eqref{n(S, T)}.

\begin{thm}
\label{mult-formula-raw}
Let $\alpha = (\alpha_i)_{i\in \mbb Z} \in \mbb{N}^{\mbb Z}$ such that $\alpha_i = \alpha_{i+n}$ for all $i\in \mbb Z$.
If $A, B \in {}^{\C}\Xi_{n,d}$  satisfy $\co(B) = \ro(A)$ and $B - \sum_{1\leq i \leq n} \alpha_i E_{\theta}^{i, i+1}$ is diagonal, then we have
\begin{equation}
\label{eBA}
e_B * e_A = \sum_{S, T} v^{2 \xi^{\C}_{A, S, T}} \, n(S, T) \begin{bmatrix} A_{S, T} \\ S
 \end{bmatrix}_{\C} e_{A_{S,T}},
\end{equation}
where the sum runs over  $S, T \in \Theta_{n}$ subject to
Condition (\ref{star}), $\ro(S) =\alpha$, $\ro(T) = \alpha^{\#}$,
$A-T+\check{T} \in \Theta_n$, $A_{S,T} \in {}^{\C}\Xi_{n,d}$. 
 \end{thm}

We make some remarks before providing the proof of this theorem.

\begin{rem}
The matrix $A_{S, T}$ depends only on $A$, $S_i$, $T_i$ for $i \in [0, r]$. Indeed,
by symmetries at $(0, 0)$ and $(r+1, r+1)$, the matrix $A_{S, T}$ is completely determined
by the entries $a'_{ij}$ for $0 \leq i \leq r+1$.
Furthermore, for $1\leq i \leq r$, the entry  $a'_{ij}$ is clearly determined by $A$ and $S_k$, $T_k$ for $k\in [0, r]$.
For $i=0$, we use (\ref{A'-2}) to get
$a'_{0 j} = a_{0j} + s_{0, j} - t_{0, j} - ( s_{0, -j} - t_{0, -j})$.
For $i=r+1$,  Condition (\ref{star}) allows us to rewrite
$a'_{r+1, j} =a_{r+1, j} + s_{r, -j} - t_{r, -j} + (s_{r, j} - t_{r, -j})$.
\end{rem}

\begin{rem}
If $S$ and $T$ satisfy
\begin{align}
\label{ST}
\sum_{j\leq k} s_{-i-1, -j} < \sum_{j \leq k} t_{ij}
\end{align}
for some $i\in [0, r]$ and $k$, then
the structural constant $n( S_i, T_i, S_{-i-1}^J)$, and hence $n(S, T)$ as well as the coefficient of $e_{A_{S, T}}$ in   (\ref{eBA}),
must be zero. This is because the summation
in $n( S_i, T_i, S_{-i-1}^J)$ is taken over upper triangular matrices
$\sigma$ such that $\ro(\sigma) = S_{-i-1}^J$ and $\co(\sigma) = T_i$, which is empty if  (\ref{ST})  is assumed.
\end{rem}

\begin{rem}
Note that Conditions (\ref{star}) and (\ref{ST}) imply that
\begin{align*}
\mbox{if $\sum_{j\geq l} s_{ij} =0$ for some $l$, then $t_{ij} = s_{-i-1, -j}$ for all $j\geq l$.}
\end{align*}
This allows to reduce the general multiplication formula to the special cases available for comparisons.
In particular, the general formula (\ref{eBA}) is compatible with the formula for Chevalley generators in ~\cite[Lemma 4.3.1]{FLLLWa}.
\end{rem}

The remainder of this section is devoted to the proof of Theorem~\ref{mult-formula-raw}.
In the proof, we shall work over a finite field $\mathbb F_q$, and all the quantum numbers and quantum binomial coefficients
(defined via the indeterminate $v$)
are understood below at the specialization $v=\sqrt{q}$ (i.e., $v^2=q$).

\subsection{Toward a proof I: type $A$ counting}
\label{Pre-I}

Let $V$ be a finite dimensional vector space over $\mathbb F_q$.
Let us fix a flag $\mbf W = (0=W_0 \subseteq W_1 \subseteq \cdots \subseteq W_m =V)$ of type $\mbf w=(w_i)_{1\leq i\leq m}$, i.e.,
$|W_i/W_{i-1}| = w_i$ for all $1\leq i \leq m$.
To a sequence $\mbf a=(a_i)_{1\leq i\leq m} \in \mbb{N}^m$, we set
\begin{align}
\label{YaW}
Y_{\mbf a} (\mbf W)= \{ U \subseteq V \big \vert  | U\cap W_i| - |U\cap W_{i-1}| = a_i,   \forall 1\leq i \leq m\}.
\end{align}
The following lemma can be found in ~\cite[Example 2.4]{Sch06}, see also Remark ~\ref{proof-lem-counting} for a proof.

\begin{lem}
\label{lemcounting}
We have
$\# Y_{\mbf a} (\mbf W) = q^{\sum_{ i > k } a_i(w_k -a_k)}\prod_{1\leq i\leq m}
\begin{bmatrix}
w_i \\ a_i
\end{bmatrix}.
$
\end{lem}

Let $\sigma$ be an upper triangular matrix such that $\ro(\sigma)=\mbf w$.
Let $\mbf t= \co(\sigma)$ and $\mathcal F_{\mbf t}(V)$ be the set of all flags in $V$ of type $\mbf t$.
Consider the set
\begin{align}
\label{FTA}
\mathcal F_{\mbf t, \sigma}^{\mbf W}
=\Big\{ \mbf F \in \mathcal F_{\mbf t}(V) \Big \vert \;  \left | \frac{W_i \cap F_j}{W_{i-1} \cap F_j + W_i \cap F_{j-1}} \right | = \sigma_{ij},   \forall 1\leq i, j\leq m \Big\}.
\end{align}
Here $\mbf W = (W_i)_{1\leq i\leq m}$ is a fixed flag of type $\mbf w$.

\begin{lem}
The cardinality of  $\mathcal F_{\mbf t, \sigma}^{\mbf W} $ is given by
\begin{align}
\label{FT}
\# \mathcal F^{\mbf W}_{\mbf t, \sigma} =
q^{\sum_{i > k, j<l} \sigma_{ij} \sigma_{kl}}
\prod_{1\leq i \leq m} \frac{[w_i]^!}{[\sigma_{ii}]^! [\sigma_{i, i+1}]^! \cdots [\sigma_{im}]^!}.
\end{align}
\end{lem}

\begin{proof}
We fix $j-1$ steps $F_1, \cdots, F_{j-1}$ such that
$\left |\frac{W_i\cap F_{j'}}{W_{i-1} \cap F_{j'} + W_i \cap F_{j'-1}}  \right |= \sigma_{i,j'}$ for all $1\leq i \leq m$ and $1\leq j' \leq j-1$.
We want to determine the number of choices of $F_j$ such that $F_{j-1} \subseteq F_j \subseteq W_j$ and
$\left | \frac{W_i\cap F_{j}}{W_{i-1} \cap F_{j} + W_i \cap F_{j-1}} \right | = \sigma_{ij}$ for all $1 \leq i \leq m$.
Via the reduction to $W_j/F_{j-1}$, this is the same as counting the number of choices of $\bar F_j$ in $W_j /F_{j-1}$ such that
$$
\left| \bar F_j \cap \frac{W_i}{F_{j-1}} \right | -
\left| \bar F_j \cap \frac{W_{i-1}}{F_{j-1}} \right | = \sigma_{ij},
\quad \forall 1\leq i \leq j.
$$
Note that $\left |\frac{W_i + F_{j-1}}{F_{j-1}} \right | - \left | \frac{W_{i-1} + F_{j-1}}{F_{j-1}} \right |
= w_i - \left | \frac{W_i\cap F_{j-1}}{W_{i-1} \cap F_{j-1}} \right | = w_i - \sum_{ j-1 \geq l} \sigma_{i l}$.
By Lemma ~\ref{lemcounting}, this number is equal to
$$
q^{\sum_{i, k: i>k} \sigma_{ij} (w_k - \sum_{j \geq  l} \sigma_{k l} )} \prod_{1\leq i \leq m}
\begin{bmatrix}
w_i - \sum_{j\leq i-1} \sigma_{ij}\\
\sigma_{ij}
\end{bmatrix}.
$$
Taking product over all $j$ and using $w_k - \sum_{j \geq  l} \sigma_{k l} =\sum_{j < l} \sigma_{kl}$, we have proved  the lemma.
\end{proof}

Let
\begin{align}
  \label{VV}
 \begin{split}
\mbf V &= ( 0= V_0 \subseteq V_1 \subseteq V_2 \subseteq V_3 \subseteq  V_4 =  V),
 \\
\mbf V' &=(0=V'_0 \subseteq V'_1 \subseteq V'_2 \subseteq \cdots \subseteq V'_m =V)
\end{split}
\end{align}
be two partial flags of $V$.  We set
\[
c_{ij} = c_{ij} (\mbf V, \mbf V') = \left | \frac{V_i \cap V'_j}{V_i \cap V'_{j-1}} \right |, \quad \forall 1\leq i \leq 3, 1\leq j\leq m.
\]
To a triple $(\mbf{s, t, t'}) $ in $(\mbb{N}^m)^3$, we define
\begin{align}
\label{Yst}
Y_{\mbf{s, t, t'}} \equiv Y_{\mbf{s, t, t'}}(\mbf V, \mbf V')
\end{align}
be the set of all subspaces $U$ such that $V_1 \subseteq U \subseteq V_3$ and subject to
the following conditions:
\begin{align}
& \left | \frac{U \cap V_2 \cap V'_j}{U \cap V_2 \cap V'_{j-1}} \right | = c_{2j} - s_j,  \tag{\ref{Yst}-i}\\
& \left | \frac{U \cap ( V_2 + V_3 \cap V'_j)}{U \cap (V_2 + V_3 \cap V'_{j-1})}  \right | = t'_j, \tag{\ref{Yst}-ii} \\
& \left | \frac{U \cap V_2 + U \cap V'_j}{U \cap V_2 + U \cap V'_{j-1}} \right | = t_j, \quad \forall 1\leq j \leq m. \tag{\ref{Yst}-iii}
\end{align}

Note that Condition (\ref{Yst}-iii) is equivalent to the following condition:
\begin{align}
\left | \frac{U \cap V'_j} {U \cap V'_{j-1}} \right | = c_{2j} - s_j + t_j, \quad \forall 1\leq j \leq m.  \tag{\ref{Yst}-iii$'$}
\end{align}
Notice also that $|U| = |V_2| - \sum_{1\leq j\leq m} s_j + \sum_{1\leq j \leq m} t'_j$, if $U\in Y_{\mbf{s, t, t'}}$.

Recall $n(\alpha, \gamma, \beta)$ from (\ref{nabc}). To a sequence of length $m$,
we can regard it as a sequence indexed by $\mbb Z$ by setting the value at the undefined positions to be zero.
So the notation $n(\mbf{s, t, t'})$ for $\mbf{s, t, t'} \in \mbb{N}^m$ is well defined.

\begin{prop}
\label{NYst}
The cardinality of $Y_{\mbf{s, t, t'}} (\mbf V, \mbf V')$ is given by
$$
\# Y_{\mbf{s, t, t'}} (\mbf V, \mbf V') =
q^{\sum_{j > l} ( c_{2j} - c_{1j}  - s_j) s_l + t'_j (c_{3l} - c_{2l} - t'_l)} n ( \mbf s, \mbf t, \mbf t')
\prod_{1\leq j\leq m}
\begin{bmatrix}
c_{3j} - c_{2j} \\
t'_j
\end{bmatrix}
\begin{bmatrix}
c_{2j} - c_{1j} \\
s_j
\end{bmatrix}.
$$
\end{prop}

\begin{proof}
We first treat the case when $V_1=0$.
We consider the following two sets:
\begin{align}
\label{Ys}
Y''_{\mbf s}  & = \{W \subseteq V_2  \Big \vert\; \left | \frac{W\cap V'_j}{W \cap V'_{j-1}} \right | = c_{2j} - s_j, \quad \forall 1 \leq j \leq m \},  \\
\label{Yt}
Y'_{\mbf t'} & = \{ T \subseteq \frac{V_3}{V_2} \Big \vert\;
  \left | T \cap \frac{V_2 + V_3 \cap V'_j}{V_2} \big / T \cap \frac{V_2 + V_3 \cap V'_{j-1}}{V_2} \right | = t'_j, \quad \forall 1\leq j \leq m \}.
\end{align}
We define a map
\begin{align}
\label{phiY}
\phi: Y_{\mbf{s, t, t'}} \longrightarrow Y''_{\mbf s} \times Y'_{\mbf t'}, \quad U \mapsto (U\cap V_2, \frac{U+V_2}{V_2}).
\end{align}
It is clear that $U\cap V_2 \in Y''_{\mbf s}$. Also we have $\frac{U+V_2}{V_2} \in Y'_{\mbf t'}$, thanks to
\begin{align*}
\begin{split}
& \left | \frac{(U+V_2) \cap (V_2 + V_3 \cap V'_j)}{V_2} \right |
-
\left | \frac{(U+V_2) \cap (V_2 + V_3 \cap V'_{j-1})}{V_2} \right |\\
& =
\left | V_2 + (U+V_2) \cap V'_j \right |
-
\left | V_2 + (U+V_2) \cap V'_{j-1} \right |  \\
& =
|(U+V_2) \cap V'_j | - | (U+V_2) \cap V'_{j-1} | - |V_2 \cap V'_j| + |V_2 \cap V'_{j-1}|\\
& =
t'_j, \quad \forall 1\leq j \leq m.
\end{split}
\end{align*}
So $\phi$ is well defined.

We shall show that $\phi$ has constant fiber and hence the cardinality of $Y_{\mbf{s, t, t'}}$
is reduced to counting  the fiber, $Y''_{\mbf s}$, and $Y'_{\mbf t'}$.
The latter two sets can be  identified with $Y_{\mbf a}(\mbf W)$ in \eqref{YaW},
where  $(\mbf a, \mbf W)$
is $\big((c_{2j}-s_j)_{1\leq j\leq m}, (V_2\cap V'_j)_{1\leq j\leq m} \big)$ for $Y''_{\mbf s}$, and
$\big( \mbf t', (\frac{V_2 + V_3 \cap V'_j}{V_2})_{1\leq j\leq m} \big)$ for $Y'_{\mbf t'}$.
So it follows by Lemma ~\ref{lemcounting} that
\begin{equation}
\label{Y''}
\begin{split}
\# Y''_{\mbf s} = q^{\sum_{j > l} (c_{2j}-s_j) s_l}
\prod_{1\leq j\leq m}
\begin{bmatrix}
c_{2j}\\ s_j
\end{bmatrix}
\quad {\rm and} \quad
\# Y'_{\mbf t'} = q^{\sum_{j > l} t_j'(c_{3l} - c_{2l} - t'_l)}
\prod_{1 \leq j \leq m}
\begin{bmatrix}
c_{3j}- c_{2j}\\ t_j'
\end{bmatrix}.
\end{split}
\end{equation}

Let $Y_{W, T}$ denote  the fiber of a fixed pair $(W, T) \in Y''_{\mbf s} \times Y'_{\mbf t'}$ under $\phi$.
We shall determine its cardinality.
We recall that the subspaces $U$ in $V$ such that $U\cap V_2 =W$ and $\frac{U+V_2}{V_2} =T$ are parametrized by $\Hom ( T, V_2/W)$.
More precisely, for any $x \in \Hom (T, V_2/W)$, we can define such a subspace
\begin{align}
\label{U(x)}
U(x) = \{ w + t + x(t) \vert  w\in W, t \in T\}
\end{align}
in $V_3$ if we fix a linear isomorphism
$V_3 = W \oplus V_2/W \oplus V_3/V_2$.
It is easy to see that for any  $x\in \Hom(T, V_2/W)$, its associated subspace $U(x)$ satisfies Condition (\ref{Yst}-i), by definition.
Observing that
\begin{align*}
\begin{split}
 \left |  \frac{(U(x) + V_2) \cap (V_2 + V_3 \cap V'_j)}{V_2} \right |
& = \left | \frac{V_2 + (U(x) + V_2) \cap V'_j}{V_2} \right  | \\
&= \left | (U(x) + V_2) \cap V'_j \right | - \left | V_2 \cap V'_j \right |, \\
  \left | U(x) \cap (V_2 + V_3 \cap V'_j) \right |
&= \left | U(x) \cap V_2 \right | + \left | (U(x) + V_2) \cap V'_j \right | - \left |V_2\cap V'_j \right |.
\end{split}
\end{align*}
This implies that $U(x) $ satisfies Condition (\ref{Yst}-ii) because $(U(x)+V_2)/V_2 =T \in Y'_{\mbf t'}$.
Thus, we have
\begin{align}
\label{Y_WT}
Y_{W, T} \cong \{ x \in \Hom (T, V_2/W) \big \vert  U(x) \ \text{satisfies Condition } (\ref{Yst}\text{-iii})\}.
\end{align}

Let $\mathcal F_{\mbf t} \equiv \mathcal F_{\mbf t}(T)$ be the set of  all partial flags in $T$ of type $\mbf t$.
Let
$$
\pi: Y_{W, T} \longrightarrow \mathcal F_{\mbf t} \equiv \mathcal  F_{\mbf t}(T)
$$
be the function defined by
\[
\pi(x) = \mbf U(x) = (U_j(x))_{1\leq j\leq m}, \ \text{where} \ U_j(x) =\frac{U(x) \cap V'_j + V_2}{V_2}, \forall 1\leq j \leq m.
\]
This is well defined because
\begin{align*}
\begin{split}
\left | U_j(x)/U_{j-1}(x)  \right |
& = \left | U(x) \cap V'_j / U(x) \cap V'_{j-1}  \right | - \left |U(x) \cap V_2 \cap V'_j/U(x)\cap V_2 \cap V'_{j-1} \right | \\
& = (c_{2j} - s_j + t_j ) - (c_{2j} - s_j) = t_j, \quad \forall 1\leq j\leq m,
\end{split}
\end{align*}
where the last equality is due to (\ref{Y_WT}) and Condition (\ref{Yst}-iii$'$).

We shall use $\pi$ and $\mathcal F_{\mbf t}(T)$ to compute the cardinality of $Y_{W, T}$.

Consider the flag $\mbf T' =(T'_j)_{1\leq j\leq m}$ associated to $T$, where
$T'_j = T \cap \frac{V_2 + V_{3} \cap V'_j}{V_2} $ for all $j$.
Clearly, we have $\mbf T' \in \mathcal F_{\mbf t'}$.
Then $\mathcal F_{\mbf t}$ admits a partition
$
\mathcal F_{\mbf t} = \sqcup_{\sigma} \mathcal F^{\mbf T'}_{\mbf t, \sigma},
$
where $\sigma$ runs over all matrices with coefficients in $\mbb{N}$ such that $\ro(\sigma) = \mbf t'$ and $\co(\sigma) =\mbf t$,
and $\mathcal F^{\mbf T'}_{\mbf t, \sigma}$ is defined in (\ref{FTA}).

Let us fix a flag $\mbf F \in \mathcal F^{\mbf T'}_{\mbf t, \sigma}$.
Since $U_j(x) \subseteq T'_j$ for all $j$, we see that the fiber $\pi^{-1} (\mbf F)$ is empty if $\sigma$ is not upper triangular.
So let us further assume that $\sigma$ is upper triangular.
By ~\cite{BLM90}, we can decompose $T= \oplus_{1 \leq i, j \leq m} Z_{ij}$ such that $|Z_{ij}| = \sigma_{ij}$ for all $1 \leq i, j \leq m$ and
\[
T'_i = \oplus_{k \leq i,  1\leq j \leq m} Z_{kj}, \quad \forall 1\leq i \leq m;
\quad
F_j = \oplus_{1\leq i \leq m, l \leq  j} Z_{i, l}, \quad \forall 1\leq j \leq m.
\]
Meanwhile, we fix a decomposition $\frac{V_2}{W} = \oplus_{1\leq j \leq m} \frac{(V_2 \cap V'_j +W)/W}{(V_2 \cap V'_{j-1} +W)/W}$.
We have
\[
\left |  \frac{(V_2 \cap V'_j +W)/W}{(V_2 \cap V'_{j-1} +W)/W}  \right | = s_j, \quad \forall 1\leq j \leq m.
\]
Under the above refinement of the spaces $T$ and $V_2/W$,
the linear maps $x$ in $\Hom(T, V_2/ W)$ can be rewritten as $x = \oplus_{1\leq i, j, l\leq m} x_{ij}^l$, where
$x_{ij}^l: Z_{ij} \to  \frac{(V_2\cap V'_l +W)/W}{(V_2 \cap V'_{l-1} +W)/W}$, for all $1\leq i, j, l\leq m$, is the restriction of $x$ to the prescribed subspaces.
Since $Z_{ij}=0$ for all $i > j$, we have $x_{ij}^l = 0$ for all $i > j$.
So $U(x)$ in (\ref{U(x)}) can be refined to be
\begin{align}
\label{U(x)-refined}
U(x) = \Big\{ w + z + \sum_{1\leq l\leq m} x^l_{ij} (z)  \big \vert\; w\in W, z\in Z_{ij}, \forall 1\leq i, j \leq m \Big\}.
\end{align}
Let us choose and fix a decomposition $V_3 \cong W \oplus V_2/W \oplus V_3/V_2$ such that
\begin{align}
\label{V'_j}
V_3 \cap V'_j = (W\cap V'_j) \oplus \big(\frac{V_2}{W} \cap V'_j \big) \oplus \big(\frac{V_3}{V_2} \cap V'_j \big), \quad \forall 1\leq j \leq m.
\end{align}
By using the descriptions (\ref{U(x)-refined}), (\ref{V'_j}),
and the fact that $U(x)$ satisfies Condition (\ref{Yst}-i), we see that the subspace $U(x)$ satisfies Condition (\ref{Yst}-iii$'$), hence (\ref{Yst}-iii),  if and only if
\begin{align*}
& \sum_{i} z_{ij} + \sum_{1\leq l\leq m} x^l_{ij} (z_{ij}) \in V'_j , \quad \forall z_{ij} \in Z_{ij}, 1\leq i, j \leq m, \\
& \sum_{i} z_{ij} + \sum_{1\leq l\leq m} x^l_{ij} (z_{ij}) \not \in    V'_{j-1} ,  \quad \forall (z_{ij})_{i\leq j}  \in \oplus_{i\leq j} Z_{ij}-\{0\}, 1\leq i, j \leq m.
\end{align*}
The first condition in the above is equivalent to $x^l_{ij}=0$ for all $l> j$.
Since
$Z_{jj} \backslash \{0\} \subseteq \frac{V_2+ V_3 \cap V'_j}{V_2} \backslash \frac{V_2 + V_3\cap V'_{j-1}}{V_2}$, we see
that $z_{jj} +  \sum_{1\leq l\leq m} x^l_{jj} (z_{jj}) \not \in V'_{j-1}$ automatically for all $z_{jj} \in Z_{jj} -\{0\}$.
Hence the second condition is equivalent to say that $\oplus_{i < j} x^j_{ij}$ is injective for each $j$, since $\oplus_{i< j} Z_{ij} \subseteq T'_{j-1} \subseteq V'_{j-1}$.
We then have
\[
\pi^{-1} (\mbf F) \cong \{ x\in \Hom( T, \frac{V_2}{W}) |  x_{ij}^l =0, \forall l > j, \oplus_{i < j} x_{ij}^j \ \text{is of rank $t_j - \sigma_{jj} $}\ \},
\]
where $\mbf F$ is a fixed flag in $\mathcal F^{\mbf T'}_{\mbf t, \sigma}$.
Observe that the number of $x_{ij}^j$  such that $\oplus_{i<j} x^j_{ij}$ is injective, hence of  rank $t_j - \sigma_{jj}$  is
$\prod_{j} \prod_{b=0}^{t_j -\sigma_{j j}-1} (q^{s_j} - q^b)$,
if one keeps in mind that the size of the corresponding matrix of $\oplus_{i<j} x_{ij}^j$, for each $j$, is  $s_j \times (t_j-\sigma_{jj})$.
The number of choices for $x_{ij}^l$ with $i , l  < j$  is
$q^{\sum_{b=1}^{j -1} \sigma_{ij} s_b}$, and the number of choices for $x_{ii}^i$ for various $i$  is $q^{\sum_{1\leq i\leq m} \sigma_{ii}  s_i}$.
Thus we have
\begin{align}
\label{pi(F)}
\pi^{-1} (\mbf F)  = q^{\sum_{i, j, l: j > l} \sigma_{ij} s_l + \sum_{1\leq i\leq m} \sigma_{ii} s_i} \prod_{j\in \mbb Z} \prod_{b=0}^{t_j - \sigma_{jj}-1} (q^{s_j}-q^b).
\end{align}
Since $\pi^{-1}(\mbf F)$ depends only on $\sigma$ and $\mbf s$,
we see that the cardinality of the fiber of the restriction $Y_{W, T; \sigma}  \to \mathcal F_{\mbf t, \sigma}^{\mbf T'}$ of $\pi$, where $Y_{W, T; \sigma} : = \pi^{-1} ( \mathcal F_{\mbf t, \sigma}^{\mbf T'})$,  is constant and given by (\ref{pi(F)}), hence
\begin{align}
\label{YWTA-raw}
\# Y_{W, T; \sigma} = \# \pi^{-1}(\mbf F)  \# \mathcal F_{\mbf t, \sigma}^{\mbf T'},
\end{align}
where $\mbf F$ is any fixed flag in $\mathcal F_{\mbf t, \sigma}^{\mbf T'}$.
By (\ref{pi(F)}), (\ref{YWTA-raw}) and (\ref{FT}), we have
\[
\# Y_{W, T; \sigma} =
q^{\sum_{j > l} \sigma_{ij} s_l + \sum_{i} \sigma_{ii} s_i + \sum_{i > k, j < l} \sigma_{ij} \sigma_{kl}}
\prod_{\substack{j \in \mbb Z \\ 0 \leq b\leq t_j - \sigma_{jj}-1}} (q^{s_j}-q^b)
\prod_{1\leq i \leq m} \frac{[t'_i]^!}{[\sigma_{ii}]^! [\sigma_{i, i+1}]^! \cdots [\sigma_{i m}]^!}.
\]
Since $Y_{W, T}$ admits a partition $Y_{W, T} = \sqcup_{\sigma} Y_{W, T; \sigma}$,
where  $\sigma$ runs over all matrices with coefficients in $\mbb{N}$ such that $\ro(\sigma) = \mbf t'$ and $\co(\sigma) =\mbf t$,
its cardinality is given by
\begin{align}
\label{YWT}
\# Y_{W, T} = \sum_{\sigma} \# Y_{W, T; \sigma} = n(\mbf{s, t, t'}).
\end{align}
Thus $\#Y_{W, T}$ is independent of the choice of $T$ and $W$. Hence $\phi$ in (\ref{phiY}) is surjective and of constant fiber,
which, together with (\ref{YWT}),  implies that
\begin{align}
\label{Y-raw}
\# Y_{\mbf{s, t, t'}} = n(\mbf{s, t, t'}) \# Y''_{\mbf s} \# Y'_{\mbf t'}.
\end{align}
Therefore, the proposition follows from (\ref{Y-raw}) and (\ref{Y''}) for the case $V_1=0$ (hence $c_{1j}=0$).

The general case  can be reduced to the case $V_1=0$ by taking the quotients with respect to $V_1$.
More precisely, the role of the pair $(\mbf V, \mbf V')$ is replaced by
the pair $(\bar{\mbf V}, \bar{\mbf V'})$ where $\bar V_i = V_i/V_1$ and $\bar V'_j = \frac{V'_j + V_1}{V_1}$ for all $1\leq i \leq 3$ and $1\leq j \leq m$.
As a consequence, we have
\[
\left | \frac{\bar V_i \cap \bar V'_j}{\bar V_i \cap \bar V'_{j-1}} \right | = c_{ij} - c_{1j} \quad \forall 1\leq i\leq 3, 1\leq j \leq m.
\]
The general case then follows from the case $V_1=0$. The proposition is proved.
\end{proof}

Let $V^*$ be the dual of $V$. We thus have a canonical pairing $\langle -, - \rangle : V^* \times V \to \mbb F_q, (f, u) \mapsto f(u)$.
Given a subspace $U\subseteq V$, we set $U^{\flat}: = \{ f\in V^* | f(u) =0, \forall u\in U\}$ to be the perpendicular of $U$ with respect to the pairing.
More generally, associated to the flags $\mbf V =(V_i)_{0\le i\le 4}$ and $\mbf V'  =(V_i')_{0\le i\le m}$ of $V$ in \eqref{VV}
we  define two flags $\tilde{\mbf V} =(\tilde V_i)_{0\le i\le 4}$ and $\tilde{\mbf V}' =(\tilde V_i')_{0\le i\le m}$ of $V^*$  by
\begin{equation}
  \label{V:flat}
\tilde V_i = V_{4 - i}^{\flat}, \qquad
\tilde{ V}'_{j} = V'^{\flat}_{m-j},  \; \text{ for } 1\leq i \leq 4,  1\leq j \leq m.
\end{equation}
To a sequence $\mbf s$, we set
\begin{equation}
  \label{flat}
\mbf s^{\flat} = (s^{\flat}_i), \qquad
 s^{\flat}_{i} = s_{m+1-i}, \; \text{ for } 1\leq i \leq m.
\end{equation}
($\mbf s^{\flat}$ is a finite analogue of $\alpha^{J}$, if we take $i\in \mbb Z$ and $m=-1$.)

\begin{prop}
\label{Y-dualY}
The assignment $U\mapsto U^{\flat}$ defines a bijection
$$
Y_{\mbf{s, t, t'}} (\mbf V, \mbf V') \cong Y_{\mbf t'^{\flat}, \mbf s'^{\flat}, \mbf s^{\flat}}(\tilde{\mbf V}, \tilde{\mbf V}'),
$$
where $\mbf s' = \mbf s + \mbf t' - \mbf t$.
\end{prop}

\begin{proof}
Recall that Condition  (\ref{Yst}-iii) is equivalent to Condition  (\ref{Yst}-iii$'$), where $c_{2j} =c_{2 j} (\mbf V, \mbf V')$,
which in turn is the same as the following condition:
\begin{align*}
\left | \frac{U + V'_j}{U + V'_{j-1}} \right | = \left | V'_j/V'_{j-1} \right | - (c_{2j} (\mbf V, \mbf V') - s_j + t_j).
\end{align*}
So we have
\begin{align*}
\left | \frac{U^{\flat} \cap \tilde V'_j}{U^{\flat} \cap \tilde V'_{j-1}} \right |
& =
\left |
\frac{(U+ V'_{m-j})^{\flat}}{(U + V'_{m+1-j})^{\flat}}
\right |
=
\left |
\frac{U + V'_{m+1-j}}{U+ V'_{m-j}}
\right |\\
& =  \left | V'_{m+1-j}/V'_{m- j} \right | - (c_{2, m+1-j} (\mbf V, \mbf V') - s_{m+1-j} + t_{m+1-j}).
\end{align*}
On the other hand,
\begin{align*}
c_{ij} (\tilde{\mbf V}, \tilde{\mbf V}')
& =  \left |
\frac{ \tilde V_i \cap \tilde V'_j}{\tilde V_i \cap \tilde V'_{j-1}}
\right |
=
\left |
\frac{V_{4-i}^{\flat} \cap V'^{\flat}_{m-j}}{V^{\flat}_{4-i} \cap V'^{\flat}_{m+1-j}}
\right |
=
\left |
\frac{ (V_{4-i} + V'_{m-j})^{\flat} } {(V_{4-i} + V'_{m+1-j})^{\flat}}
\right |\\
&
=
\left |
\frac{V_{4-i} + V'_{m+1-j}}{V_{4-i} + V'_{m-j}}
\right |
=
\left |
\frac{V'_{m+1-j}}{V'_{m-j}}
\right |
-
\left |
\frac{V_{4-i} \cap V'_{m+1-j}}{V_{4-i} \cap V'_{m-j}}
\right | \\
&
=
\left |
\frac{V'_{m+1-j}}{V'_{m-j}}
\right |
- c_{4-i, m+1-j}(\mbf V, \mbf V').
\end{align*}
Thus we have
\begin{align}
\label{Uflat-iii}
\begin{split}
\left | \frac{U^{\flat} \cap \tilde V'_j}{U^{\flat} \cap \tilde V'_{j-1}} \right |
& =
c_{2j} (\tilde{\mbf V}, \tilde{\mbf V}') + s_{m+1-j} - t_{m+1-j}
 = c_{2j} (\tilde{\mbf V}, \tilde{\mbf V}')  + s'^{\flat}_j - t'^{\flat}_j,
\end{split}
\end{align}
where the last equality is due to the definition of $\mbf s'$.
Condition (\ref{Uflat-iii}) is Condition (\ref{Yst}-iii) for $U^{\flat}$  in
$Y_{\mbf t'^{\flat}, \mbf s'^{\flat}, \mbf s^{\flat}}(\mbf V^{\flat}, \mbf V'^{\flat})$.

We note that
\[
|(U\cap (V_2 + V_3 \cap V'_j))^{\flat}| = |U^{\flat}| + |V_2^{\flat} \cap V'^{\flat}_{j} | - |U^{\flat} \cap V^{\flat}_2 \cap V'^{\flat}_{j} |.
\]
So we have
\begin{align}
\begin{split}
\left |
\frac{ U^{\flat} \cap \tilde V_2 \cap \tilde V'_{m+1-j}} {U^{\flat} \cap \tilde V_2 \cap \tilde V'_{m-j}}
\right |
=
c_{2, m+1-j} (\tilde{\mbf V}, \tilde {\mbf V}') - t'_j
=
c_{2, m+1-j} (\tilde{\mbf V}, \tilde {\mbf V}') - t'^{\flat}_{m+1-j}.
\end{split}
\end{align}
Change the index  $j \leftrightarrow m+1-j$, we see that $U^{\flat}$ satisfies Condition (\ref{Yst}-i) for $Y_{\mbf t'^{\flat}, \mbf s'^{\flat}, \mbf s^{\flat}}(\mbf V^{\flat}, \mbf V'^{\flat})$.

Tracing backward the above argument, we see that $U^{\flat}$ satisfies Condition (\ref{Yst}-ii) for
$Y_{\mbf t'^{\flat}, \mbf s'^{\flat}, \mbf s^{\flat}}(\mbf V^{\flat}, \mbf V'^{\flat})$ can be deduced from Condition (\ref{Yst}-i) for $U$.

So the map defined by $U\mapsto U^{\flat}$ is well defined, and $(U^{\flat})^{\flat}=U$ implies that the map is a bijection. The proposition follows.
\end{proof}

\begin{cor}
\label{weak-n-dual}
If $Y_{\mbf{s, t, t'}} (\mbf V, \mbf V') \neq \O $, then
$n(\mbf s, \mbf t, \mbf t') = n(\mbf t'^{\flat}, \mbf s'^{\flat}, \mbf s^{\flat}) $ where $\mbf s - \mbf t = \mbf s' - \mbf t'$.
\end{cor}

\begin{proof}
The proof of Proposition ~\ref{Y-dualY} also shows that there are bijections
\[
Y''_{\mbf s} \cong Y'_{\mbf s^{\flat}}
\quad \mbox{and} \quad
Y'_{\mbf t'} \cong Y''_{\mbf t'^{\flat}},
\]
for the auxiliary sets in (\ref{Ys}) and (\ref{Yt}) for $Y_{\mbf{s, t, t'}}$ in the left and
$Y_{\mbf t'^{\flat}, \mbf s'^{\flat}, \mbf s^{\flat}}(\mbf V^{\flat}, \mbf V'^{\flat})$ in the right.
The corollary then follows from  (\ref{Y-raw}).
\end{proof}

\subsection{Toward a proof II: type $C$ counting}

In this section, we assume that $V$  is an even dimensional vector space over $\mbb F_q$ equipped with a non-degenerate symplectic form.
As in the previous section,  we fix a flag $\mbf W =(W_i)_{0\leq i\leq m}$ of type $\mbf w = (w_i)_{1 \leq i \leq m}$.
We require that  $m=2r+1$, and $\mbf W$ is isotropic, i.e., $W_i^{\perp} = W_{m - i}$ for all $0 \leq i \leq m$.
For a fixed $j$, we define
\begin{equation}
  \label{YjVW}
Y_j \equiv Y_j (V, \mbf W; u) =\{ U \subseteq V | U\subseteq U^{\perp}, |U| = u, U \subseteq W_j, U\cap W_{j-1} =0 \}.
\end{equation}

\begin{lem}
The cardinality of $Y_j$ in \eqref{YjVW}  is given by
\begin{align}
\label{Yj}
\# Y_j =
\begin{cases}
q^{u \sum_{1\leq j' \leq j-1} w_{j'}}  \begin{bmatrix}w_j\\ u \end{bmatrix}, & \mbox{if} \ j < r+1,\\
q^{u \sum_{1\leq j' \leq j-1} w_{j'}} \prod_{0 \leq i \leq u-1} \frac{[w_j - 2i]}{[i+1]}, & \mbox{if} \ j = r+1, \\
q^{u \sum_{1\leq j' \leq j-1} w_{j'} - u(u-1)/2} \begin{bmatrix}w_j\\ u \end{bmatrix}, & \mbox{if} \ j > r+1.
\end{cases}
\end{align}
\end{lem}

\begin{proof}
Assume that $u =1$. Since any one-dimensional subspace in $V$ is automatically isotropic, we see that
\begin{align}
\label{m=1}
Y_j (V, \mbf W; 1) = \frac{q^{|W_j|} -1}{q-1} - \frac{q^{|W_{j-1}|}-1}{q-1} = q^{|W_{j-1}|} \frac{q^{|W_j/W_{j-1}|} -1}{q-1}
=q^{\sum_{i \leq j-1} w_i} [w_j].
\end{align}
This is exactly (\ref{Yj}) for $u=1$.

Now we treat the general case. We introduce a new set
\[
\tilde Y_j(V, \mbf W; k) = \{ \mbf U=(U_i)_{1\leq i\leq k} \big \vert \,  U_i \subseteq U_{i+1} ,\quad U_i \in Y_j(V, \mbf W; i) \}.
\]
In particular $\tilde Y_j(V, \mbf W; 1) = Y_j(V, \mbf W; 1)$.
Consider the projection
\[
\pi_k: \tilde Y_j (V, \mbf W; k+1) \to \tilde Y_j (V, \mbf W; k), \quad (U_i)_{1\leq i \leq k+1} \mapsto (U_i)_{1\leq i \leq k}.
\]
Fix a flag $(U_i)_{1\leq i \leq k}$ in $\tilde Y_j (V, \mbf W; k)$.
Since $U_{k}$ is isotropic, $U_{k}^{\perp}/U_{k}$ inherits a non-degenerate symplectic form from $V$.
We set  $\mbf W^{k}=(W^{k}_i)_{1\leq i\leq m}$ where  $W^{k}_i = \frac{W_i \cap U_{k}^{\perp} + U_{k}}{U_{k}}$.
Observe that
$|W^{k}_i/ W^{k}_{i-1} | = w_i - k \delta_{ij} - k \delta_{i, m + 1-j} $ for all $1\leq i \leq m$.
Then we have a bijection:
$$
\pi_k^{-1} (U_i)_{1\leq i\leq k} \simeq Y_j (U_{k}^{\perp}/U_{k}, \mbf W^{k}; 1), \quad (U_i)_{1\leq i\leq k+1} \mapsto U_{k+1}/U_{k}.
$$
Hence it follows by (\ref{m=1}) that
\begin{align}
\label{pik-fiber}
\# \pi_k^{-1} (U_i)_{1\leq i\leq k} =
\begin{cases}
q^{\sum_{j'\leq j-1} w_{j'}} [w_j -k] & \forall j <  r+1, \\
q^{\sum_{j' \leq j-1} w_{j'} - (1- \delta_{j, r+1}) k} [w_j-(1+ \delta_{j, r+1}) k],  &  \forall  j\geq r+1.
\end{cases}
\end{align}
This implies that $\pi_k$ is surjective with constant fiber.
Applying repeatedly (\ref{pik-fiber}), we obtain
\begin{align}
\label{tYj}
\# \tilde Y_j (V, \mbf W; u)
=
\begin{cases}
q^{u \sum_{j'  \leq j-1} w_{j'}} \prod_{0\leq k \leq u-1} [w_j - k], & \forall j < r+1, \\
q^{u \sum_{j'  \leq j-1} w_{j'}  - (1- \delta_{j, r+1})  {u \choose 2}}  \prod_{0\leq k \leq u - 1} [w_j - (1+ \delta_{j, r+1} )k], & \forall j\geq r+1.
\end{cases}
\end{align}
The natural projection from
$\tilde Y_j (V, \mbf W; u) $ to $Y_j(V, \mbf W; u)$ is surjective
and its fiber is the set of all complete flags in an $u$-dimensional space over $\mbb F_q$.
Since  the cardinality of the latter set is $[u]!$,
we have $\# Y_j(V, \mbf W; u) = \# \tilde Y_j (V, \mbf W; u)/ [u]!$, from which  the lemma follows.
\end{proof}

Recall  the set $Y_{\mbf a}(\mbf W)$ in (\ref{YaW}), and recall in addition that $m=2r+1$ and $\mbf W$ is isotropic in this section.
Let
\begin{equation}
  \label{Yasp}
Y_{\mbf a}^{\mrm{sp}}(\mbf W) =\{ U \in Y_{\mbf a} (\mbf W) | U\subseteq U^{\perp}\}.
\end{equation}

\begin{lem}
\label{YspaW}
The cardinality of $Y^{\mrm{sp}}_{\mbf a} (\mbf W)$ in \eqref{Yasp} is given by
\[
\#Y^{\mrm{sp}}_{\mbf a} (\mbf W)
=
q^{\xi_{\mbf a}} \prod_{j=1}^r
\begin{bmatrix}
w_j \\ a_j
\end{bmatrix}
\prod_{j=r+2}^m
\begin{bmatrix}
w_j-a_{m +1-j} \\ a_j
\end{bmatrix}
\prod_{i=0}^{a_{r+1} -1} \frac{[w_{r+1} - 2i]}{[i+1]},
 \]
 where
 $\xi_{\mbf a} = \sum_{j < l } (w_j -a_j) a_l - \sum_{j > l ,  j+l > m +1} a_j a_{m+1-l} - \sum_{j>r+1} \frac{a_j (a_j -1)}{2}$.
\end{lem}

\begin{proof}
We set $\mbf a^k = (a_1, \cdots, a_k, 0 \cdots, 0)$ for all $0\leq k \leq m$.
We define a map
\[
\pi_k : Y^{\mrm{sp}}_{\mbf a^{k+1}}(\mbf W) \longrightarrow Y^{\mrm{sp}}_{\mbf a^{k}} (\mbf W), U_{k+1} \mapsto U_{k+1} \cap W_k.
\]
Fix a point $U_k \in  Y^{\mrm{sp}}_{\mbf a^{k}} (\mbf W)$, we can form the symplectic space $U_k^{\perp}/U_k$ since $U_k$ is isotropic.
We set $\mbf W^k =( W^k_i)_{1\leq i \leq m}$ with $W^k_i = (W_i\cap U_k^{\perp})/U_k$ for all $1\leq i\leq m$ in $U_k^{\perp}/U_k$.
As before, we have
$|W^k_i/W^k_{i-1}| = w_i - \mbf a^k_{i} - \mbf a^k_{m+1-i}$ for all $1\leq i\leq m$,
where $\mbf a^k_i$ is the $i$-th entry of $\mbf a^k$.
We then have  a bijection
$$
\pi_k^{-1}(U_k) \simeq Y_{k+1} (U_k^{\perp}/U_k, \mbf W^k; a_{k+1}), \quad
U_{k+1} \mapsto U_{k+1}/U_k.
$$
By using (\ref{Yj}), we have
\begin{align}
\label{pij}
\# \pi_{j-1}^{-1}(U_{j-1}) =
\begin{cases}
q^{\sum_{1\leq j' < j < r+1} (w_{j'} - a_{j'}) a_j} \begin{bmatrix} w_j\\ a_j \end{bmatrix}, & \mbox{if} \ j < r+1,\\
q^{\sum_{1\leq j' < j < r+1} (w_{j'} - a_{j'} ) a_{r+1}} \prod_{0 \leq i \leq a_{r+1}-1} \frac{[w_{r+1} - 2i]}{[i+1]}, & \mbox{if} \ j= r+1, \\
q^{\xi_j}
\begin{bmatrix} w_j - a_{m+1-j}\\ a_j\end{bmatrix}, & \mbox{if}\ j > r+1,
\end{cases}
\end{align}
where $\xi_j = \sum_{1\leq j'  <j, j>r+1}  (w_{j'}  - a_{j'})  a_j - \sum_{j' < j,  j + j' > m+1} a_j a_{m+1-j'} - \sum_{j> r+1} a_j(a_j-1)/2$.
So the lemma follows from (\ref{pij}) and that
$\#Y_{\mbf a}^{\mrm{sp}}(\mbf W) = \prod_{j=1}^{m} \# \pi_{j-1}^{-1}(U_{j-1})$ for fixed $U_{j-1} \in Y_{\mbf a^{j-1}}^{\mrm{sp}}(\mbf W)$.
\end{proof}

\begin{rem}
\label{proof-lem-counting}
Lemma ~\ref{lemcounting} is a special case of Lemma ~\ref{YspaW} for $\mbf a= (a_1, \ldots, a_{r-1}, 0, \ldots, 0)$.
\end{rem}

Recall $Y_{\mbf{s, t, t'}}(\mbf V, \mbf V')$, $Y''_{\mbf s}$ and $Y'_{\mbf t'}$ from (\ref{Yst}), (\ref{Ys}) and (\ref{Yt}), respectively.
For $\mbf V$ in \eqref{VV}, we further assume that $V_3 \subseteq V_2^{\perp}$.
Let
\begin{equation}
  \label{Ystt}
Y^{\mrm{sp}}_{\mbf{s, t, t'}}(\mbf V, \mbf V')
=\{ U \in Y_{\mbf{s, t, t'}}(\mbf V, \mbf V') | U \subseteq U^{\perp} \}.
\end{equation}

\begin{prop}
  \label{NYsp}
The cardinality of $Y^{\mrm{sp}}_{\mbf{s, t, t'}}(\mbf V, \mbf V')$ in \eqref{Ystt} is given by
\begin{align*}
\# Y^{\mrm{sp}}_{\mbf{s, t, t'}}(\mbf V, \mbf V')
= & q^{\xi_{\mbf{s, t, t'}}} n({\bf s}, {\bf t}, {\bf t'})
\prod_{j=1}^m \begin{bmatrix} c_{2 j}-c_{1j}\\ s_j \end{bmatrix}
\prod_{j=1}^r \begin{bmatrix} c_{3j} -c_{2j}\\ t_j' \end{bmatrix}
  \\
& \qquad \qquad \cdot \prod_{j=r+2}^m \begin{bmatrix} c_{3j} - c_{2 j} - t_{m+1-j}' \\ t_j' \end{bmatrix}
\prod_{i=0}^{t_{r+1}' -1} \frac{[c_{3, r+1} - c_{2, r+1} - 2i]}{[i+1]},
\end{align*}
where
$$
\xi_{\mbf{s, t, t'}}=\sum_{j > l} (c_{2j} - c_{1j} - s_j) s_l + \sum_{j < l } (c_{3j} - c_{2j} - t_j') t'_l - \sum_{j  > l > m +1-j} t_j' t_{m+1- l}' - \sum_{j>r+1} t_j' (t_j'-1)/2.
$$
\end{prop}

\begin{proof}
We  treat the $V_1=0$ case first.
Let $Y'^{\mrm{sp}}_{\mbf t'}  = \{ T\in Y'_{\mbf t'} | T \subseteq T^{\perp}\}$ where $T^{\perp}$ is taken inside the symplectic space $V_2^{\perp}/V_2$.
We have the following commutative diagram.
\[
\begin{CD}
Y^{\mrm{sp}}_{\mbf{s, t, t'}}(\mbf V, \mbf V') @>\phi^{\mrm{sp}}>>
Y''_{\mbf s} \times Y'^{\mrm{sp}}_{\mbf t'}  \\
@VVV @VVV \\
Y_{\mbf{s, t, t'}}(\mbf V, \mbf V') @> \phi >>
Y''_{\mbf s} \times Y'_{\mbf t'},  \\
\end{CD}
\]
where $\phi$ is defined in (\ref{phiY}), the vertical maps are inclusions and $\phi^{\mrm{sp}}$ is the map induced from $\phi$.
Now for a pair $(W, T)$ in  $Y''_{\mbf s} \times Y'^{\mrm{sp}}_{\mbf t'}$, the fiber of $(W, T)$ under $\phi$ is contained in
$ Y^{\mrm{sp}}_{\mbf{s, t, t'}}(\mbf V, \mbf V')$ (here we freely use identifications under the inclusions).
This implies that $\phi^{\mrm{sp}}$ is of constant fiber and its fiber is the same as that of $\phi$.
Thus $\# Y^{\mrm{sp}}_{\mbf{s, t, t'}}(\mbf V, \mbf V') = \# \phi^{-1}(T, W) \cdot  \# Y''_{\mbf s} \cdot \# Y'^{\mrm{sp}}_{\mbf t'} $.
We know $\#Y''_{\mbf s}$ by (\ref{Y''}), and  $\# \phi^{-1}(T, W)=n(\mbf{s, t, t'})$ by (\ref{YWT}), and  by Lemma ~\ref{YspaW},
\[
\# Y'^{\mrm{sp}}_{\mbf t'}
=
q^{\xi_{\mbf t'}}
\prod_{j=1}^r \begin{bmatrix} c_{3j} - c_{2j} \\ t_j' \end{bmatrix}
\prod_{j=r+2}^n \begin{bmatrix} c_{3j} - c_{2j} - t_{m+1-j}' \\ t_j' \end{bmatrix}
\prod_{i=0}^{t_{r+1}'-1} \frac{[c_{3, r+1} -c_{2, r+1} - 2i]}{[i+1]},
\]
where $\xi_{\mbf t'}=\sum_{j < l} (c_{3j} - c_{2j} - t'_j) t'_l -\sum_{j > l >  m +1-j} t'_j t'_{m +1-l} - \sum_{j>r+1} t_j' (t_j'-1)/2$.
The proposition for $V_1=0$ follows from  these computations.

The general case can be reduced to the $V_1=0$ case by considering the reduction $U \mapsto U/V_1$ as in the proof of Proposition ~\ref{NYst}.
The proposition is proved.
\end{proof}

The rest of the section is a description of the duals of the set $Y^{\mrm{sp}}_{\mbf a} (\mbf W)$ and $Y^{\mrm{sp}}_{\mbf{s, t, t'}}(\mbf V, \mbf V')$.
Recall $Y^{\mrm{sp}}_{\mbf a}(\mbf W)$ from \eqref{Yasp}. We set
\[
^{\mrm{sp}} Y_{\mbf a} (\mbf W ) = \{ U\in Y_{\mbf a} (\mbf W) | U^{\perp} \subseteq U\}.
\]

\begin{lem}
The assignment $U\mapsto U^{\perp}$ defines a bijection
$^{\mrm{sp}} Y_{\mbf a} (\mbf W) \cong Y^{\mrm{sp}}_{\hat{\mbf a}} (\mbf W)$, where $\hat{\mbf a} = (\hat{a}_j)_{1\leq j\leq m} $, $\hat a_j = w_j - a_{m+1-j}$ for all $j$.
\end{lem}

\begin{proof}
Since
$
|(U\cap W_i)^{\perp} | = |U^{\perp} | + |W_{m-i}| - |U^{\perp} \cap W_{m-i} |,
$
we have
\[
|U^{\perp} \cap W_{m+1-i} | - | U^{\perp} \cap W_{m-i} | = w_{m+1-i} - a_i =  \hat a_{m+1-i}.
\]
The lemma is proved.
\end{proof}

Recall $Y^{\rm sp}_{\mbf{s, t, t'}}(\mbf V, \mbf V')$ from \eqref{Ystt}. We assume further that there is a non-degenerate symplectic form on $V_2$.
We set
\[
^{\mrm{sp}} Y_{\mbf{s, t, t'}}(\mbf V, \mbf V') = \{ U\in Y_{\mbf{s, t, t'}}(\mbf V, \mbf V') \big \vert  U^{\perp} \subseteq U\},
\]
where $U^{\perp}$ is taken with respect to the form on $V_2$.
Note that to define $^{\mrm{sp}} Y_{\mbf{s, t, t'}}(\mbf V, \mbf V') $, a form on the whole space $V$ is not needed.
Recall the notations $\mbf s^{\flat}$  from \eqref{flat} and  $\tilde{\mbf V}$ from \eqref{V:flat}.

\begin{prop}
\label{Ysp-dual}
We have a bijection $Y^{\mrm{sp}}_{\mbf{s, t, t'}}(\mbf V, \mbf V') \cong \ ^{\mrm{sp}} Y_{\mbf{t'^{\flat}, s'^{\flat}, s^{\flat}}} (\tilde{\mbf V}, \tilde{\mbf V}')$.
\end{prop}

\begin{proof}
In  the definition of $Y^{\mrm{sp}}_{\mbf{s, t, t'}}(\mbf V, \mbf V')$,  the assumption that $V$ is equipped with a non-degenerate symplectic form
is not essential. It is still well defined if we assume that $V$ is equipped with a possibly degenerate symplectic form such that $V_2^{\perp} = V_3$,
hence $V_3/V_2$ inherited a non-degenerate symplectic form from that of $V$.
With this point, the bijection follows readily from the proof of Proposition ~\ref{Y-dualY}.
\end{proof}

\subsection{Step 1 of the proof: the piece $Z_{S, T}$}

Recall the setting from Theorem ~\ref{mult-formula-raw}. Suppose that
$e_B * e_A = \sum g^{A'}_{B, A} e_{A'}$, we shall narrow down those $A'$ which could have nonzero structure constants and interpret the latter as the cardinality of a given set.

Let us fix a periodic chain $L$ of symplectic lattices in $\X^{\fc}_{n,d} (\ro(B))$ (cf. \eqref{L-L}), for $B =(b_{ij}) \in {}^{\C}\Xi_{n,d}$.
We consider the set
\[
Z = \{ L'' \in \X^{\fc}_{n,d} (\co(B)) \big \vert  (L, L'') \in \mathcal O_B\}.
\]
Recall $\alpha =(\alpha_i)_{i \in \ZZ}$ is given in Theorem~\ref{mult-formula-raw}.

\begin{lem}
The set $Z$ admits a characterization as follows.
\[
Z = \Big\{ L'' \in \X^{\fc}_{n,d} (\co(B)) \Big \vert
L_{i-1} \subseteq L_i'' \subseteq L_{i+1},
L_{i-1}
\overset{\alpha_{n-i}}{\subseteq}
L_{i-1} + L''_{i-1} \overset{b_{ii}}{\subseteq}  L_i \cap L''_i \overset{\alpha_i}{\subseteq} L_i,   \forall i \in \mbb Z \Big\}.
\]
Here and below $\overset{d}{\subseteq}$ and $\overset{d}{\supseteq}$ denote inclusions of codimension $d$.
\end{lem}

\begin{proof}
Since $(L, L'') \in \mathcal O_B$, we can decompose $V = \oplus_{i, j \in \mbb Z} M_{ij}$ as $\mbb F_q$-vector spaces
such that $\ve M_{ij}= M_{i-n,j-n}$, $|M_{ij}| = b_{ij}$,
$L_i = \oplus_{\substack{k \leq i \\ j \in \mbb Z}} M_{k j}$ and $L''_j = \oplus_{\substack{ i\in \mbb Z \\ l \leq j}} M_{i, l}$, for all $i, j\in \mbb Z$.
By the definition of $B$, we have
$$
L_{i} =  \bigoplus_{ k \leq i, j \leq i} M_{k j} \oplus M_{i, i+1},
\qquad
L''_{i} = \bigoplus_{k \leq i, j\leq i} M_{k j} \oplus M_{i+1, i}.
$$
In particular, we have $L_{i-1} \subseteq L''_i \subseteq L_{i+1}$. The second condition in the characterization of $Z$
is clear.  Hence $Z$ is included in the set on the right hand side of the equation in the lemma.

The other direction of inclusion follows from the definition.
\end{proof}

Let us fix a second chain $L'$ in $\X^{\fc}_{n,d}$. We set
\begin{align}
\label{cij}
c_{ij} = c_{ij}(L, L') = \left  |\frac{L_i \cap L_j'}{L_i \cap L_{j-1}'} \right |, \quad \forall i, j \in \mbb Z.
\end{align}
We formulate the following condition:
\begin{equation}
\begin{split}
& L_{i}'' \cap  L_{i} \cap  L'_j    \overset{c_{ij}-s_{ij}}{ \supseteq}    L_{i}''\cap  L_{i} \cap  L_{j-1}',\\
& L_{i}'' \cap  L_{i}+ L_{i}''\cap  L'_j    \overset{t_{ij}}{ \supseteq}    L_{i}''\cap  L_{i}+  L_{i}''\cap L_{j-1}', \quad \forall i, j\in \mbb Z.
\end{split}
\tag{\ref{cij}${}_{i j}$}
\end{equation}
Associated to $S, T\in \Theta_{n}$, we define a subset $Z_{S, T}$ of $Z$ as
\begin{equation}
  \label{ZST}
Z_{S, T} =
\{ L'' \in Z \big \vert L''_i\, (\forall i\in \mbb Z) \text{ satisfies Condition ({\ref{cij}${}_{i j}$})} \}.
\end{equation}
Clearly, $\{ Z_{S, T} \}$ forms a partition of $Z$.

\begin{lem}
\label{condition-star}
If $Z_{S, T}$ in \eqref{ZST} is nonempty, then  $S$ and $T$ must  satisfy Condition (\ref{star}).
\end{lem}

\begin{proof}
For any $i, j \in \mbb Z$, we have
\begin{equation*}
\begin{split}
\left | \frac{L_i'' \cap L_j'}{L_i'' \cap L_{j-1}'} \right | - \left | \frac{L_i ''\cap L_i \cap L_j'}{L_i''\cap L_i \cap L_{j-1}'} \right |
& = \left | \frac{L_i'' \cap L_j'}{L_i ''\cap L_i \cap L_j'} \right | - \left  | \frac{L_i'' \cap L_{j-1}'}{L_i''\cap L_i \cap L_{j-1}'} \right | \\
& = \left | \frac{L_i ''\cap L_i+ L_i'' \cap L_j'}{L_i ''\cap L_i} \right |- \left | \frac{L_i ''\cap L_i+L_i'' \cap L_{j-1}'}{L_i''\cap L_i} \right | \\
&= \left | \frac{L_i ''\cap L_i+ L_i'' \cap L_j'}{L_i ''\cap L_i+L_i'' \cap L_{j-1}'}  \right |.
\end{split}
\end{equation*}
Thus by the assumption of the pair  $(S, T)$ in ({\ref{cij}${}_{i j}$}), we have
\begin{align}
\label{i2}
\left | \frac{L_i'' \cap L_j'}{L_i'' \cap L_{j-1}'} \right |
=
c_{ij} - s_{ij} + t_{ij}, \quad \forall i, j \in \mbb Z.
\end{align}

Moreover, for all $j \in \mbb Z$, we have
\begin{equation}
\label{i2dual}
\begin{split}
\left | \frac{L''_i \cap L_j'}{L''_i \cap L_{j-1}'} \right |
& = \left | \frac{(L''_i \cap L_{j-1}')^{\#}}{(L''_i \cap L_j')^{\#}} \right |
= \left | \frac{L_{-i -1}''+ L_{-j}'}{L_{-i -1}'' + L_{-j-1}'} \right | \\
& = \left  | \frac{L_{-i -1}''+ L_{-j}'}{L_{-i -1}''} \right | - \left  | \frac{L_{-i -1}'' + L_{-j-1}'}{L_{-i -1}''} \right |\\
& = \left  | \frac{ L_{-j}'}{L_{-j-1}'} \right | - \left  | \frac{L_{-i -1}''\cap L_{-j}'}{L_{-i -1}''\cap  L_{-j-1}'} \right |  \\
& \overset{(\ref{i2})}{=} \sum_{k \in \mbb Z} a'_{k , -j} - (c_{-i -1, -j} + t_{-i -1, -j} - s_{-i -1, -j} )  \\
& = \sum_{k \leq i} a_{k j}' - t_{-i -1,-j} + s_{-i -1, -j}\\
& = c_{ij} - t_{-i -1, -j} + s_{-i -1, -j}.
\end{split}
\end{equation}
By combining (\ref{i2}) and (\ref{i2dual}),  Condition (\ref{star}) holds  for $S$ and $T$.
\end{proof}

By (\ref{i2}), we actually have
\begin{equation*}
s_{ij} = \left |\frac{L_{i-1} + L_i \cap L'_j}{L_{i-1} + L_i \cap L'_{j-1}} \right |
-
\left | \frac{L_i'' \cap (L_{i-1} + L_i \cap L'_j)} {L''_i \cap ( L_{i-1} + L_i \cap L'_{j-1})} \right |.
\end{equation*}


\begin{lem}
\label{ZST-2}
If $Z_{S,T}$ in \eqref{ZST}  is  nonempty, then    $\ro(S)= \alpha$ and $\ro(T) = \alpha^{\#}$ where $\alpha^{\#}$ is defined in (\ref{alpha2}).
\end{lem}

\begin{proof}
By definition, we have
\[
\sum_{j\in \mbb Z} t_{ij}
= \sum_{j\in \mbb Z} \left | \frac{ L''_{i} \cap  L_{i}+ L_{i}''\cap  L'_j }{L''_{i} \cap  L_{i}+ L_{i}''\cap  L'_{j-1} } \right |
= \left | \frac{L''_i}{L''_i\cap L_i} \right |
=\alpha_{n - (i+1)} = \alpha_{-(i+1)}, \quad \forall i\in \mbb Z,
\]
where the second equality follows from the observation that the sequence $(L''_{i} \cap  L_{i}+ L_{i}''\cap  L'_j)_{j\in \mbb Z}$ stabilizes
to be  $L''_i$ for $j\gg 0$ and to be $L''_i \cap L_i$ for $-j\gg 0$.

Similarly, we have
\begin{align*}
\begin{split}
\sum_{j\in \mbb Z} s_{ij}
& = \sum_{j\in \mbb Z} \left  |\frac{L_i \cap L_j'}{L_i \cap L_{j-1}'} \right |
-  \left | \frac{L_{i}'' \cap  L_{i} \cap  L_j'}{L_{i}'' \cap  L_{i} \cap  L'_{j-1}} \right | \\
& = \sum_{j\in \mbb Z} \left | \frac{L_i\cap L'_j}{L''_i \cap L_i \cap L'_j} \right | - \left | \frac{L'_i \cap L'_{j-1}}{L''_i \cap L_i \cap L'_{j-1}} \right | \\
& = \left | \frac{L_i}{L''_i \cap L_i} \right |  = \alpha_i, \quad \forall i\in \mbb Z.
\end{split}
\end{align*}
The lemma follows.
\end{proof}

The following lemma justifies the relevance of the partition $Z = \bigsqcup Z_{S, T}$.

\begin{lem}
\label{BAST}
Assume  $(L, L'') \in \mathcal O_B$ and $(L'', L') \in \mathcal O_{A}$.
If  $L'' \in Z_{S, T}$, then $(L, L') \in \mathcal O_{A_{S, T}}$.
\end{lem}

\begin{proof}
Suppose that $(L, L') \in \mathcal O_{A'}$ with $A'=(a'_{ij})$.
We have
\begin{align}
\label{a-a}
a'_{ij} - a_{ij} =
\left | \frac{L_i\cap L'_j}{L_i \cap L'_{j-1}}  \right |
- \left | \frac{L''_i\cap L'_j}{L''_i \cap L'_{j-1}} \right |
- \left | \frac{L_{i-1} \cap L'_j}{L_{i-1} \cap L'_{j-1}} \right |
+ \left | \frac{L''_{i-1} \cap L'_j}{L''_{i-1} \cap L'_{j-1}}  \right |, \quad \forall i, j\in \mbb Z.
\end{align}

By (\ref{a-a}), (\ref{i2}) and the definition of $c_{ij}$ in \eqref{cij}, we obtain
\begin{equation*}
\label{Aij}
a_{ij}'=
a_{ij} + (s_{ij}-t_{ij}) - (s_{i-1,j} - t_{i-1,j}), \quad \forall i, j \in \mbb Z.
\end{equation*}
Therefore $A'=A_{S, T}$.
The lemma is proved.
\end{proof}

Summarizing, we have established the following.

\begin{prop}
Retain the assumptions in Theorem ~\ref{mult-formula-raw}. Then we have
$e_B * e_A = \sum  \# Z_{S, T} \;e_{A_{S, T}}$,
where the sum runs over all pairs $(S, T)$  in $\Theta_{n}$ subject to Condition (\ref{star}),  $\ro(S) =\alpha$ and $\ro(T) = \alpha^{\#}$,
$A-T+\check{T} \in \Theta_n$, and $A_{S,T} \in {}^{\C}\Xi_{n,d}$.
\end{prop}

\begin{proof}
For any pair $(L, L')$ of symplectic chains, we have
\begin{align*}
e_B * e_A (L, L')
&= \sum_{L''} e_B(L, L'') e_A(L'', L')
=\sum_{L'' \in Z} e_A (L'', L') \\
&= \sum_{S, T} \sum_{L'' \in Z_{S, T}} e_A (L'', L')
= \sum_{S, T} \# Z_{S, T} e_{A_{S, T}} (L, L'),
\end{align*}
where the last equality is due to Lemma ~\ref{BAST}.
Lemmas ~\ref{condition-star} and ~\ref{ZST-2} guarantee that we can impose the conditions on $S$ and $T$ in the proposition.
\end{proof}

\subsection{Step 2 of the proof: counting $Z_{S, T}$}

For $0\leq i \leq r$, we introduce an index set $\mbb Z(i)$ to be the subset of $\mbb Z$ consisting of all integers $k \in [-i-1, i]$ mod $n$.
Recall from the previous section that we fix a pair $(L, L') \in \mathcal O_{A_{S, T}}$.
Let $Z_{S, T}^{[0, i]}$ be the set of symplectic lattice chains $(L''_k)_{k \in \mbb Z(i)}$ such that $L''_k$ satisfies
$L_{k-1} \subseteq L''_k \subseteq L_{k+1}$ and the  conditions ({\ref{cij}${}_{k j}$}) for all $k\in \mbb Z(i)$.
We then have the following tower of  projections:
\[
\begin{CD}
Z_{S, T} @>\pi_r >> Z_{S, T}^{[0, r-1]} @>\pi_{r-1}>> Z_{S, T}^{[0, r-2]} @>\pi_{r-2}>> \cdots @>\pi_1>> Z_{S, T}^{[0,0]} @>\pi_0>> \bullet
\end{CD}
\]
where $\pi_i (L''_k)_{k\in \mbb Z(i)} = (L''_k)_{k\in \mbb Z(i-1)}$ for $0\leq i \leq r$.
We shall show that each $\pi_i$ is of constant fiber and hence the cardinality of $Z_{S, T}$ is a product of the cardinality of the fibers of $\pi_i$ for $0\leq i\leq r$.

So we focus on computing the cardinality of the fiber of $\pi_i$.
We set
\begin{equation}
\label{ZSTi}
Z^i_{S, T} (\mbf L^i)= \pi_i^{-1}(\mbf L^i)
\end{equation}
to be the fiber of $\mbf L^i =(L^i_k)_{k\in \mbb Z(i-1)}$ under $\pi_i$ for all $0\leq i \leq r$.

Due to the periodicity, i.e., $\ve L_m =L_{m-n}$,
the set $Z^i_{S, T}(\mbf L^i)$, for $1\leq i \leq r$,  is in bijection with the set of pairs $(L''_{-i-1}, L''_i)$ of lattices  in $V_F$ such that
\begin{align}
  \label{LSTi}
\begin{split}
&L''^{\#}_i = L''_{-i-1},\;\;
L^i_{i-1} +  L_{i-1} \subseteq L''_i \subseteq L_{i+1}, \;\;
L_{-i-2} \subseteq L''_{- i -1} \subseteq L_{-i}, \\
&\text{Condition ({\ref{cij}${}_{k j}$}) is satisfied}, \; \forall j\in \mbb Z \text{ and }  k = -i-1, i.
\end{split}
\end{align}
(Here $V_F$ is a $2d$-dimensional  symplectic space over $F= \mbb F_q((\ve))$.)
The third condition is redundant since it is implied by the first two conditions.
The first condition in ({\ref{cij}${}_{k j}$}) for $k = -i-1$ is equivalent to the following condition:
\begin{align}
L''_i \cap (L_i + L_{i+1}\cap L'_j) \overset{s_{-i-1,-j}}{\supseteq} L''_i \cap (L_i + L_{i+1}\cap L'_{j-1}), \quad \forall  j\in \mbb Z.
\tag{\ref{LSTi}${}_{ij}$}
\end{align}

Indeed, we have
\begin{align*}
& (L''_i \cap (L_i + L_{i+1}\cap L'_j))^{\#}
= L''_{-i-1} + (L_{-i-1} \cap (L_{-i-2} + L'_{-j-1})) \\
& =L''_{-i-1} + (L_{-i-2} + L_{-i-1} \cap L'_{-j-1})
= L''_{-i-2} + L_{-i-1} \cap L'_{-j-1}.
\end{align*}
Thus
\begin{align*}
\left |
\frac{L''_i \cap (L_i + L_{i+1}\cap L'_j)}{L''_i \cap (L_i + L_{i+1}\cap L'_{j-1})}
\right |
&=
\left |
\frac{(L''_i \cap (L_i + L_{i+1}\cap L'_{j-1}))^{\#}}{(L''_i \cap (L_i + L_{i+1}\cap L'_{j}))^{\#}}
\right |\\
& =
\left |
\frac{L''_{-i-2} + L_{-i-1} \cap L'_{-j}}{L''_{-i-2} + L_{-i-1} \cap L'_{-j-1}}
\right |\\
& =
\left |
\frac{L_{-i-1} \cap L'_{-j}}{L_{-i-1} \cap L'_{-j-1}}
\right |
-
\left |
\frac{L''_{-i-1}\cap L_{-i-1} \cap L'_{-j}}{L''_{-i-1}\cap L_{-i-1} \cap L'_{-j-1}}
\right | \\
& = c_{-i-1, -j} - (c_{-i-1, -j} - s_{-i-1, -j}) = s_{-i-1, -j}.
\end{align*}
The second condition of ({\ref{cij}${}_{k j}$}) for $k=-i-1$ is equivalent to Condition (\ref{star}) by the argument in the proof of Lemma ~\ref{condition-star}.

So the above analysis allows us to identify $Z^i_{S, T}(\mbf L^i)$ for $1\leq i \leq r-1$ with the set of all lattices $\mathcal L$ such that
\begin{align*}
\begin{split}
 \mathcal L^{\#} \subseteq \mathcal L, \;
 L^i_{i-1} + L_{i-1} \subseteq \mathcal L \subseteq L_{i+1}, \;
 \mathcal L\ \text{satisfies Conditions ({\ref{cij}${}_{i j}$}) and ({\ref{LSTi}${}_{ij}$})} \ \forall   j\in \mbb Z.
\end{split}
\end{align*}

Recall that we can find a lattice $\mathcal M$ such that $L_{n-1} \subseteq \mathcal M \subseteq L_n$ and
$\mathcal M/\ve\mathcal M$ admits a non-degenerate symplectic form over $\mbb F_q$ from that  on $V_F$ over $F$.
To a lattice $\mbb L \in V_F$, we can define $\overline{\mbb L} = (\mbb L \cap \mathcal M + \ve \mathcal M) / \ve \mathcal M$.
Thus, $\overline{\mathcal M}$, $\overline L$,  and $\overline{L'}$ are well defined.

Recall the notation $Y_{\mbf{s, t, t'}}(\mbf V, \mbf V')$ from (\ref{Yst}). To a matrix $M$, we write $M_i$ for its $i$-th row vector.
By Lemma~\cite[Lemma 3.3.3]{FLLLWa}, the assignment $\mathcal L \mapsto \overline{\mathcal L}$ defines a bijection
\begin{align}
\label{ZiST}
Z^i_{S, T} (\mbf L^i) \longrightarrow Y_{S_i, T_i, S_{-i-1}^J}(\mbf V^i, \mbf V'),  \quad \forall 1\leq i \leq r-1,
\end{align}
where $\mbf V^i =(V^i_a)_{0\le a \le 4}$,
$V^i_1= \overline{L^i_{i-1}+L_{i -1}}$, $V^i_k = \overline{L_{i+k-2}}$ for $k=2, 3$,  and $V'_j = \overline{L'_{j}}$ for all  $j\in \mbb Z$.
Observe that
\begin{align*}
c_{1j} (\mbf V^i, \mbf V') & = c_{i-1, j} (\overline L, \overline{L'}) - s_{-i, -j}, \\
c_{kj}(\mbf V^i, \mbf V') & = c_{i + k -  2, j}(\overline L, \overline{L'}), \quad \forall 2\leq k\leq 3.
\end{align*}
where $c_{kj}(L, L')$ is the $c_{kj}$ in (\ref{cij}) (with $i$ replaced by $k$)  for the pair $(L, L')$.
So by Proposition ~\ref{NYst}, we have, for $1\leq i\leq r-1$,
\begin{align}
\label{Zk}
\begin{split}
\# Z^i_{S, T}(\mbf L^i)
&=
q^{\sum_{j > l} ( a'_{i j} - s_{-i, -j}   - s_{i j}) s_{i l} + s_{-i -1, -j} (a'_{i +1, l}  - s_{-i-1, -l} )}  \\
& \qquad \quad  \cdot n ( S_{i}, T_i, S_{-i-1}^J)  \cdot
\prod_{j\in \mbb Z}
\begin{bmatrix}
a'_{i+1, j} \\
s_{-i-1, -j}
\end{bmatrix}
\begin{bmatrix}
a'_{i j} - s_{-i, -j}\\
s_{i j}
\end{bmatrix}.
\end{split}
\end{align}

The $i=r$ case in (\ref{ZiST}) is excluded since
we do not know if $\overline{\mathcal L}$ is isotropic, while in other cases it is automatically isotropic because it is contained in the isotropic subspace $\overline{L_{r}}$.
By imposing the condition that $\overline{\mathcal L}$ is isotropic,  we have a bijection (recall \eqref{Ystt})
\[
Z^r_{S, T} (\mbf L^r) \cong Y^{\mrm{sp}}_{S_r, T_r,  S_{-r-1}^J}(\mbf V^r, \mbf V'), \quad \mbox{if} \ r\geq 1.
\]
So by Proposition ~\ref{NYsp}, we have, for $r\geq 1$,
\begin{align}
\begin{split}
\label{Zr}
\# Z^r_{S, T} (\mbf L^r)
&=
q^{\xi_r} n(S_r, T_r, S^J_{-r-1})
\prod_{j\in \mbb Z} \begin{bmatrix} a'_{r j} - s_{-r, -j} \\ s_{r j} \end{bmatrix}
\prod_{j\leq r} \begin{bmatrix} a'_{r+1,j} \\  s_{-r-1, -j} \end{bmatrix} \\
& \qquad
\cdot \prod_{j\geq r+2} \begin{bmatrix} a'_{r+1,j}  - s_{-r-1, -n+j }  \\ s_{-r-1, -j} \end{bmatrix}
\prod_{i=0}^{s_{-r-1, -r-1}  -1} \frac{[ a'_{r+1, r+1}  - 2i]}{[i+1]},
\end{split}
\end{align}
where
\begin{align*}
\begin{split}
\xi_r
=
\sum_{j > l} (a'_{rj} - s_{-r, -j} - s_{rj}) s_{r l} + \sum_{j >l  }  s_{-r-1, -j} (a'_{r+1, l} -s_{-r-1, - l} )  \\
- \sum_{j  > l > n -j} s_{-r-1, -j} s_{-r-1, - n + l}
- \sum_{j> r+1} s_{-r-1, -j} (s_{-r-1, -j}-1)/2.
\end{split}
\end{align*}

The remaining case is $i=0$.
The only difference from the other cases is that $\mbf L^0=\bullet$ does not play a role here.
So we can identify $Z^0_{S, T} \equiv Z^0_{S, T}(\mbf L^0)$ with the set of lattices $\mathcal L$ such that
\begin{align}
\label{Z0r}
 L_{-1} \subseteq \mathcal L\subseteq L_1,   \mathcal L^{\#} \subseteq \mathcal L,
\text{ and } \mathcal L \ \mbox{satisfies Conditions ({\ref{cij}${}_{0 j}$}) and ({\ref{LSTi}${}_{0j}$}) for $j\in \mbb Z$}.
\end{align}

If  $L_{-1} \subseteq \mathcal M$, it implies that $\mathcal M^{\#} \subseteq L^{\#}_{-1} = L_0$, and hence
$(L_0 \cap \mathcal M)^{\#} = L_{-1} + \mathcal M^{\#} \subseteq L_0$.
If further $\mathcal M^{\#} \subseteq \mathcal M$, then $(L_0 \cap \mathcal M)^{\#} \subseteq L_0 \cap \mathcal M$.
On the other hand, if $(L_0 \cap \mathcal M)^{\#} \subseteq L_0 \cap \mathcal M$, we have
$\mathcal M^{\#} \subseteq (L_0 \cap \mathcal M)^{\#} \subseteq L_0 \cap \mathcal M \subseteq \mathcal M$.

So $Z^0_{S, T}$ is in bijection with the set of all lattices $\mathcal L$ such that
\begin{align*}
L_{-1} \subseteq \mathcal L \subseteq L_1,
(L_0 \cap \mathcal L)^{\#} \subseteq L_0 \cap \mathcal L,
 \mbox{$\mathcal L$ satisfies Conditions ({\ref{cij}${}_{0 j}$}) and ({\ref{LSTi}${}_{0j}$}) for  $j\in \mbb Z$.}
\end{align*}

The symplectic form on $V_F$ descends to a non-degenerate symplectic $\mbb F_q$-form on $L_0/L_{-1}=L_0/ L^{\#}_0$ (but not $L_1/L_{-1}$).
See for example \cite[0.8]{Lu03}. It is assumed that $L_0$ is of
even volume, but  it can be generalized to arbitrary volume. Here the volume of a lattice is defined as the volume form on V attached to the symplectic form.

With this information, the condition $(L_0 \cap \mathcal L)^{\#} \subseteq L_0 \cap \mathcal L$  is equivalent to
\begin{align*}
\mbox{$U \subseteq L_1/L_{-1}$ such that $(U\cap (L_0/L_{-1}))^{\perp} \subseteq U\cap (L_0/L_{-1})$ in $L_0/L_{-1}$.}
\end{align*}
So
\begin{align}
\label{Z0ST}
Z^0_{S, T}
=\ ^{\mrm{sp}} Y_{S_0, T_0, S^J_{-1}} (\mbf V^0, \mbf V') \cong Y^{\mrm{sp}}_{S_{-1}, T_{-1}, S^J_0}(\mbf V^{-1}, \mbf V'),
\end{align}
where the last equality is due to Proposition ~\ref{Ysp-dual} via $\mathcal L \mapsto \mathcal L^{\#}$.

Now by applying Proposition ~\ref{NYsp}, we have
\begin{align}
\label{Z0}
\# Z^0_{S, T} =
q^{\xi_0} n(S_{-1}, T_{-1}, S_0^J)
\prod_{j\in \mbb Z} \begin{bmatrix} a'_{-1,j} \\ s_{-1, j} \end{bmatrix}
\prod_{j<0} \begin{bmatrix} a_{0j}' \\ s_{0, -j}  \end{bmatrix}
\prod_{j>0} \begin{bmatrix} a'_{0j} - s_{0 j} \\ s_{0 , - j}  \end{bmatrix}
\prod_{i=0}^{s_{00} - 1} \frac{[a'_{00} - 2i]}{[i+1]},
\end{align}
where
\begin{align*}
\xi_0=\sum_{j > l} (a'_{-1, j}  - s_{-1, j}) s_{-1, l}
+ \sum_{j > l } s_{0, -j} (a'_{0l} - s_{0, -l})
-\sum_{j > l >  -j} s_{0, -j}  s_{0 l}
- \sum_{j>0}  s_{0, -j} (s_{0, -j} -1)/2.
\end{align*}

By Corollary ~\ref{weak-n-dual}, the term $n(S_{-1}, T_{-1}, S_0^J)$ in (\ref{Z0})  can be replaced by $n(S_0, T_0, S^J_{-1})$.
Hence, by (\ref{Z0}),
(\ref{Zk}), (\ref{Zr}) for $r \geq 1$, we see that all $\pi_i$ are surjective with constant fiber, and so
\[
\# Z_{S, T} = \prod_{i=0}^r  \# Z^i_{S, T}(\mbf L^i)= q^{\xi^{\C}_{A, S, T}} n(S, T) \begin{bmatrix} A_{S, T} \\ S
 \end{bmatrix}_{\C}, \quad \forall r\geq 1.
\]
Therefore, Theorem ~\ref{mult-formula-raw} is proved for $r\geq 1$.

Now we deal the remaining case $r=0$ for Theorem ~\ref{mult-formula-raw}, which we will put $|_{r=0}$ whenever appropriate to emphasize this special case.
In this case $Z_{S, T}|_{r=0}$ can be identified with the set of lattices such that
\begin{align}
\label{Z00}
\begin{split}
& L_{-1} \subseteq \mathcal L\subseteq L_1, \quad \mathcal L^{\#} \subseteq \mathcal L \subseteq \ve^{-1} \mathcal L^{\#}, \\
& \mathcal L \ \mbox{satisfies Conditions ({\ref{cij}${}_{0j}$}) and ({\ref{LSTi}${}_{0j}$}) for $j\in \mbb Z$}.
\end{split}
\end{align}
Note that the only difference of the above description of $Z_{S, T}|_{r=0}$ from (\ref{Z0r}) is the extra condition $ \mathcal L \subseteq \ve^{-1} \mathcal L^{\#}$ which holds automatically for  the $r\geq 1$ case.

If $\mathcal L$ satisfies the conditions in (\ref{Z00}), we have
$
\mathcal L \subseteq L_1 \cap \ve^{-1} \mathcal L^{\#} = \ve^{-1} (L_0^{\#}\cap \mathcal L^{\#})= \ve^{-1} (L_0 + \mathcal L)^{\#}.
$
By taking $\#$ and multiplying $\ve^{-1}$ on the first condition in (\ref{Z00}), we get $L_0 \subseteq \ve^{-1} \mathcal L^{\#}$, so we get
$L_0 \subseteq \ve^{-1} (L^{\#}_0 \cap \mathcal L^{\#})$.
Hence, we have $L_0 + \mathcal L \subseteq \ve^{-1} (L_0 + \mathcal L)^{\#}$.
On the other hand,  if $L_0 + \mathcal L \subseteq \ve^{-1} (L_0 + \mathcal L)^{\#}$, we have
$\mathcal L \subseteq L_0 + \mathcal L \subseteq \ve^{-1} (L_0 + \mathcal L)^{\#} \subseteq \ve^{-1} \mathcal L^{\#}$,
that is, $\mathcal L \subseteq \ve^{-1} \mathcal L^{\#}$.

So the description (\ref{Z00}) is the same as  the set of all lattices $\mathcal L$ such that
\begin{align*}
& L_{-1} \subseteq \mathcal L \subseteq L_1, \quad
(L_0 \cap \mathcal L)^{\#} \subseteq L_0 \cap \mathcal L, \quad
L_0 + \mathcal L \subseteq \ve^{-1} (L_0 + \mathcal L)^{\#},
\\
& \mbox{$\mathcal L$ satisfies Conditions ({\ref{cij}${}_{0j}$}) and ({\ref{LSTi}${}_{0j}$}) for $j\in \mbb Z$.}
\end{align*}

Since $L_1 = \ve^{-1} L_0^{\#}$, the quotient space $L_1/L_0$ inherits a non-degenerate symplectic $\mbb F_q$-form from $V_F$.
Via the descent  $\mathcal L \mapsto \mathcal L + L_0/L_0 =: \overline{\mathcal L}$,
the condition $L_0 + \mathcal L \subseteq \ve^{-1} (L_0 + \mathcal L)^{\#}$ is  equivalent to the following condition
\begin{align*}
\overline{\mathcal L} \ \mbox{is isotropic in $L_1/L_0$}.
\end{align*}
So $Z_{S, T}|_{r=0}$ is in bijection with the subset $Y|_{r=0}$ of
$^{\mrm{sp}} Y_{S_0, T_0, S^J_{-1}} (\mbf V^0, \mbf V')$ in (\ref{Z0ST}) defined by the above condition.
The computation of $\#  Y|_{r=0}$ is similar to  that of $^{\mrm{sp}} Y_{S_0, T_0, S^J_{-1}} (\mbf V^0, \mbf V')$, i.e.,
\begin{align}
\label{r=0-a}
\# Z_{S, T}|_{r=0} =\#  Y|_{r=0}= n(S_0, T_0, S_{-1}^J) \# Y'^{\mrm{sp}}_{S^J_{-1}} \# ^{\mrm{sp}} Y''_{ S_0},
\end{align} where
$Y'^{\mrm{sp}}_{S^J_{-1}}$ and  $^{\mrm{sp}} Y''_{S_0}$ are auxiliary sets attached to $^{\mrm{sp}} Y_{S_0, T_0, S^J_{-1}} (\mbf V^0, \mbf V')$.
We have
\begin{align}
\label{r=0-b}
Y'^{\mrm{sp}}_{S^J_{-1}}
=
q^{\xi_{S^J_{-1}}} \prod_{j < 1} \begin{bmatrix} a'_{1j}\\ s_{-1, -j} \end{bmatrix}
\prod_{j > 1} \begin{bmatrix} a'_{1j} - s_{-1, -2 + j} \\ s_{-1, -j} \end{bmatrix}
\prod_{i=0}^{s_{-1, -1} -1} \frac{[a'_{11} - 2i]}{[i+1]},
\end{align}
where $\xi_{S^J_{-1}} = \sum_{j < l} (a'_{1j} - s_{-1, -j} ) s_{-1, - l} - \sum_{j > l > - j} s_{-1, -j} s_{-1, l} - \sum_{j > 1} s_{-1, -j} (s_{-1, -j} -1) /2$.
Similarly, we have
\begin{align}
\label{r=0-c}
\# ^{\mrm{sp}} Y''_{S_0}
= \# Y'^{\mrm{sp}}_{S_0^J} =
q^{\xi_{S_0^J}} \prod_{j< 0} \begin{bmatrix} a'_{0j} \\ s_{0, -j} \end{bmatrix}
\prod_{j > 0} \begin{bmatrix} a'_{0 j} - s_{0 j} \\ s_{0, - j} \end{bmatrix}
\prod_{i=0}^{s_{00} -1} \frac{[a'_{00} - 2i]}{[i+1]},
\end{align}
where
$\xi_{S_0^J} = \sum_{j < l} (a'_{0j} - s_{0, -j} ) s_{0, -l} - \sum_{j > l > -j} s_{0, -j} s_{0, l} - \sum_{j > 1} s_{0, -j} (s_{0, -j} -1) /2$.
By (\ref{r=0-a})-(\ref{r=0-c}), we complete the proof of the $r=0$ case and hence complete the proof of Theorem ~\ref{mult-formula-raw}.

\section{The quantum group $\Kc_n$ via the multiplication formula}
 \label{chap:Kjj}

In this section, we obtain a monomial basis  for the convolution algebra $\Sj$
based on the multiplication formula obtained in Section~\ref{chap:multi}.
We observe a stabilization property from this multiplication formula, which allows us
to construct a limit algebra $\Kc_n$ for the family of convolution algebras $\{\Sj\}_d$.
We construct a monomial basis and canonical basis for $\Kc_n$, as well as a
surjective homomorphism from $\Kc_n$ to $\Sj$.

The index set  in  \eqref{Mdn} for bases of $\Sj$ is used
in the formulation of the multiplication formula as well as in further applications in this and later Sections.

\subsection{A monomial basis of the convolution algebras}
  \label{secstandbasis}

Recall $\MX_{n,d}$ from \eqref{Mdn} and the bijection ${}^{\C}\Xi_{n,d} \leftrightarrow \MX_{n,d}$ from \eqref{bijection}.
We first reformulate Theorem~\ref{mult-formula-raw} using the index set $\MX_{n,d}$.
Set
\begin{align}
a'_{ij} &= a_{ij} + (s_{ij} + s_{-i, -j}) - (t_{ij} + t_{-i, -j}),
 \notag \\
\xi^{\mathfrak b}_{A, S, T} &=
\sum_{\substack{-r-1 \leq i \leq r\\ j> l}} (a_{ij}' -s_{ij}) s_{i l}
- \sum_{\substack{i=-r-1, 0 \\ j< i }} \frac{s_{i j} (s_{i j} + 1)}{2}
 \label{xiAST2}
 \\
&\qquad\qquad \qquad\quad
 - \Big(\sum_{\substack{-r \leq i \leq -1 \\ j > l}} + \sum_{\substack{ i =  -r -1, 0 \\  2i -l > j > l }}   \Big) s_{-i, - j} s_{i l},\notag
\\
 \begin{bmatrix}
A_{S, T}  \\
 S
 \end{bmatrix}_{\mathfrak b}
 & =
\prod_{(i, j)\in \mathcal J}
\begin{bmatrix} a'_{ij} \\ s_{ij} \end{bmatrix}
\begin{bmatrix} a'_{ij} - s_{ij} \\ s_{-i, -j} \end{bmatrix}
\prod_{\substack{i= 0, -r-1\\ 0 \leq k \leq s_{ii}-1 }} \frac{[a'_{ii} - 2k-1]}{[k +1]}.
 \label{AST2}
\end{align}

\begin{thm}
\label{mult-formula11}
Let $\alpha = (\alpha_i)_{i\in \mbb Z} \in \mbb N^{\mbb Z}$ such that $\alpha_i = \alpha_{i+n}$ for all $i\in \mbb Z$.
If $A, B \in \MX_{n,d}$  satisfy $\co(B) = \ro(A)$ and $B - \sum_{1\leq i \leq n} \alpha_i E_{\theta}^{i, i+1}$ is diagonal, then we have
\begin{equation}
\label{eBA-2}
e_B * e_A = \sum_{S, T} v^{2 \xi^{\mathfrak b}_{A, S, T}} \; n(S, T) \begin{bmatrix} A_{S, T} \\ S\end{bmatrix}_{\mathfrak b} e_{A_{S,T}},
\end{equation}
where the sum runs over all $S, T \in \Theta_{n}$ subject to Condition \eqref{star},
  $\ro(S) =\alpha$ and $\ro(T) = \alpha^J$,
$A-T+\check{T} \in \Theta_n$, and $A_{S,T} \in \MX_{n,d}$.
\end{thm}

 \begin{proof}
 Follows by Theorem \ref{mult-formula-raw} and the bijection between ${}^{\C}\Xi_{n,d}$ and $\MX_{n,d}$.
\end{proof}

We set
 \begin{equation}
  \label{hST2}
\begin{split}
&h_{S,T}
= \sum_{\overset{i\in [1,n]}{j>l}} a_{ij}s_{il} -  \sum_{\overset{i\in [1,n]}{j<l}} a_{ij} t_{il} - \sum_{i\in [1,n]} a_{ij} s_{ij}
+ \sum_{i\in [1,n]} \alpha_i \alpha_{i-1} - \frac{1}{2} \sum_{i\in [1,n]} \alpha_{n-i} \alpha_i\\
&+ \frac{1}{2} \sum_{i\in [1,n]} \alpha_i\alpha_{n-1-i}
 -2 \sum_{\overset{i\in [1,n]}{j>l}} t_{ij}s_{il} - 2 \sum_{\overset{i\in [1,n]}{j>l}} s_{i-1,j}s_{il}
+ 2 \sum_{\overset{i\in [1,n]}{j>l}} t_{i-1,j}s_{il} -2 \sum_{\overset{i\in [-r,-1]}{j>l}} s_{-i,-j}s_{il} \\
&- 2 \sum_{\overset{i=-r-1,0}{l < j < 2i-l}} s_{-i,-j}s_{il}
-\sum_{\overset{i=-r-1,0}{j<i}} s_{ij}(s_{ij}-1) + \frac{3}{2} \sum_{j>0} t_{0j}- \frac{3}{2} \sum_{j<r+1} s_{r+1,j} - \frac{1}{2} \sum_{j<0} s_{0j}\\
& +  \frac{1}{2} \sum_{j>r+1} t_{r+1,j}  -\alpha_0 -\frac{1}{2} \alpha_{r+1} + \frac{1}{2}\alpha_{-r-2} - s_{r+1,r+1} -\frac{1}{2}s_{00} - \frac{1}{2} t_{r+1,r+1}.
 \end{split}
\end{equation}
We now reformulate Theorem~\ref{mult-formula11} in terms of the standard basis elements $[A]$. 

\begin{thm}
\label{mult-stand-basis2}
Let $\alpha = (\alpha_i)_{i\in \mbb Z} \in \mbb N^{\mbb Z}$ such that $\alpha_i = \alpha_{i+n}$ for all $i\in \mbb Z$.
If $A, B \in \MX_{n,d}$  satisfy $\co(B) = \ro(A)$ and $B - \sum_{1\leq i \leq n} \alpha_i E_{\theta}^{i, i+1}$ is diagonal,
then we have
\begin{equation*}
[B] * [A] = \sum_{S, T} v^{h_{S,T}} \; n(S, T) \begin{bmatrix}
  A_{S,T}\\ S
\end{bmatrix}_{\mathfrak b} [A_{S,T}],
\end{equation*}
where the sum runs over all $S, T \in \Theta_{n}$ subject to Condition \eqref{star},
  $\ro(S) =\alpha$ and $\ro(T) = \alpha^J$,
$A-T+\check{T} \in \Theta_n$, and $A_{S,T} \in \MX_{n,d}$.
\end{thm}

\begin{proof}
  By the definition of $[A]$ and Theorem~\ref{mult-formula11}, we obtain a multiplication formula as stated in the theorem,
  where
  \begin{equation}
   \label{MST}
  h_{S,T} =  d_{A_{S,T}} - d_A - d_B + 2 \xi^{\mathfrak b}_{A,S,T}.
  \end{equation}
Recall $\xi^{\mathfrak b}_{A,S,T}$ from \eqref{xiAST2}, and note that
  \begin{equation*}
    \begin{split}
      d_{A_{S,T}} & - d_A - d_B = - \sum_{\overset{i\in [1,n]}{j\geq l}} a_{ij}s_{il}  - \sum_{\overset{i\in [1,n]}{j<l}} a_{ij}t_{il}
      +\sum_{i\in [1,n]} \alpha_i\alpha_{i-1} - \frac{1}{2} \sum_{i\in [1,n]} \alpha_{n-i}\alpha_i \\
      + \frac{1}{2} & \sum_{i\in [1,n]} \alpha_i\alpha_{n-1-i}
      +\frac{3}{2} \sum_{j>0} t_{0j}- \frac{1}{2}\sum_{j\leq 0} s_{0j} + \frac{1}{2} \sum_{j>r+1} t_{r+1,j}-\frac{3}{2} \sum_{j<r+1} s_{r+1,j}\\
       - & \frac{1}{2}t_{r+1,r+1}-s_{r+1,r+1} - \alpha_0 -\frac{1}{2}\alpha_{r+1}+\frac{1}{2} \alpha_{-r-2} .
    \end{split}
  \end{equation*}
Then $h_{S,T}$ can be rewritten in the desired form by a direct calculation.
\end{proof}

Define a partial order $\leq_{\text{alg}}$ on $\MX_{n,d}$ in exactly the same manner as the one on $^{\C} \Xi_{n,d}$.
Again, by ``lower terms (than $[A'])$'',
we refer to the terms $[C]$ with $C <_{\text{alg}} A'$, $\ro(C) = \ro(A')$ and $\co(C) = \co(A')$.

\begin{prop}\label{leading}
  Let $\alpha = (\alpha_i)_{i\in \mbb Z} \in \mbb N^{\mbb Z}$ be such that $\alpha_i = \alpha_{i+n}$ for all $i$.
Let $A= (a_{ij})$ be such that $a_{ij}= 0$ if $|j-i| \geq m$ for some $m >1$,
$a_{i,i+m-1} \geq \alpha_{i-1}$ and $a_{i,i-m+1} \ge  \alpha_{n-i-1}$ for all $i$.
Let $B \in \MX_{n,d}$ be such that $\co(B) = \ro(A)$ and  $B - \sum_{1\leq i \leq n} \alpha_i E_{\theta}^{i, i+1}$ is diagonal. Then we have
\begin{equation}
[B] * [A] = [A'] + {\rm lower\ terms},
\end{equation}
where $A'=(a_{ij}')$ is given by
\begin{equation}\label{A'}
a_{ij}'= \left\{\begin{array}{ll}
0  & {\rm if} \  |j-i| > m, \\
a_{ij} & {\rm if} \ |j-i| < m-1, \\
a_{ij}+ \alpha_i     & {\rm if}\ j= i+m,\\
a_{ij}- \alpha_{i-1} & {\rm if}\ j= i+m-1,\\
a_{ij}+ \alpha_{n-i}     & {\rm if}\ j= i-m,\\
a_{ij}- \alpha_{n-i-1} & {\rm if}\ j= i-m+1.
\end{array}
\right.
\end{equation}
\end{prop}

\begin{proof}
  For such a given $A=(a_{ij})$, $A_{S,T}$ is a leading term if   $S, T$ satisfy that
  \[s_{i,i+m} = \alpha_i,\quad t_{i, i-m+1} = \alpha_{n-1-i} = s_{-i-1, m-i-1},\quad \forall i\in [0, r+1].
  \]
 For such $S, T$, the matrix $A_{S,T}$ is identified with $A'$ with entries given by (\ref{A'}). It remains to determine the
 leading coefficient. Note that
 $\xi^{\mathfrak b}_{A, S, T} =0, d_{A'}- d_A - d_B =0$, and
  \begin{equation*}
    n(S, T) = q^{\sum_{j>l} s_{-i-1, -j}s_{i, l}} = 1,
     \qquad
     \begin{bmatrix}
      A_{S,T}\\
      S
    \end{bmatrix}_{\mathfrak b}
    = \prod_i \begin{bmatrix}
      a_{i, i+m}'\\
      \alpha_i
    \end{bmatrix}
    \begin{bmatrix}
      a_{i, i-m}'\\
      s_{-i,m-i}
    \end{bmatrix} =1.
  \end{equation*}
Hence it follows by Theorem~\ref{mult-stand-basis2} and \eqref{MST} that the coefficient for $[A']$  is $1$.
\end{proof}

\begin{thm}\label{monomial-basis}
For any $A \in \MX_{n,d}$, there exist finitely many tridiagonal matrices $B(i)$ such that
\begin{align}
m_A': &=  \prod_{i\geq 0}^{\longleftarrow}  [B(i)] = [A] + \text{lower terms} \in \Sj,
  \label{eq3} \\
m_A: &=  \prod_{i\geq 0}^{\longleftarrow}  \{B(i)\} = [A] + \text{lower terms} \in \Sj.
 \label{eq32}
\end{align}
Here the products are taken in a reverse order (such as $\cdots  * [ B(1)] * [B(0)]$ in \eqref{eq3}).
\end{thm}

\begin{proof}
For any $A = (a_{ij}) \in \MX_{n,d}$, fix $m \in \mbb N$ such that $a_{ij}=0$ for all $|j-i| >m$.
Let $B(0)=(b_{ij})$ be a diagonal matrix such that $b_{ii}= \sum_{j=i-m}^{i+m} a_{ji}$.
For any $i\in [1,m]$, we define a tridiagonal matrix $B(i)$ such that $B(i) - \sum_{1\leq j \leq n} \alpha_{ij} E_{\theta}^{j, j+1}$ is diagonal
and $\alpha_{ij} = \sum_{k =  j+i-m}^{j} a_{k,j+i}$.
By the property of the entries of $A$, we have
$\alpha_{i,n-j}= \sum_{k =  j}^{j-i+m} a_{k,j-i}$.
By Proposition~\ref{leading}, we have the identity (\ref{eq3}).
Recall by definition of $\{B\}$ in Section \ref{rec-sj}, we have $\{B\} =[B] + \text{lower terms}$.
One further checks that the products of the lower terms arising from $\{B(i)\}$ produce terms lower than $[A]$,
and hence $m_A$ has the desired property.
The theorem follows.
\end{proof}

For each $A \in \MX_{n,d}$, we fix one such choice of  $m_A$ as given in the proof of Theorem~\ref{monomial-basis}.

\begin{cor}
The set $\{ m_A \vert A  \in \MX_{n,d}\}$ forms a basis of $\Sj$.
\end{cor}
We shall call the basis $\{ m_A \vert A  \in \MX_{n,d}\}$ a {\em monomial basis} of $\Sj$.

\subsection{Monomial basis and canonical basis of $\Kc_n$}

For $A \in {\rm Mat}_{\mbb Z \times \mbb Z}(\mbb Z)$ we set
\begin{equation}
  {}_p\!A= A +p I,\quad \forall p \in 2 \mathbb N,
\end{equation}
where $I$ is the identity matrix.
Set
$\widetilde{\MX}_n$
to be the set of all matrices  $A =(a_{ij})_{i,j\in \mbb Z}$  such that
$a_{ij} \in \mbb N$,  for all $i\neq j$,  $a_{ii} \in \mbb Z$,
$a_{ij}=a_{-i, -j}=a_{i+n, j+n}$ and $\sum_{i=1}^n \sum_{j\in \mbb Z} a_{ij}$ is finite.
This is a generalization of the set $\MX_n :=\sqcup_d \MX_{n,d}$
by dropping the positivity condition on the diagonal entries.
For any given $A \in \widetilde{\MX}_{n}$, we have ${}_p\!A \in {\MX}_{n}$ for $p\gg 0$.

For an indeterminate $v'$, we introduce a commutative  ring $\mathcal R = \mbb Q(v)[v', v'^{-1}].$
We have the following stabilization result.
\begin{prop} \label{prop6.1}
  Suppose that $A_1, A_2,\ldots, A_l \ (l \geq 2)$ are  matrices in $\widetilde{\MX}_n$
such that ${\rm co}(A_i)={\rm ro}(A_{i+1})$.
There exist $Z_1, \ldots, Z_m\in \widetilde{\MX}_n$, $G_j(v,v')\in \mathcal R$ and $p_0\in \mbb N$ such that
$$[{}_p A_1] * [{}_pA_2] * \cdots *[{}_p A_l]=\sum_{j=1}^mG_j(v,v^{-p})[{}_p Z_j],\quad
\forall p \in 2 \mbb N, p\geq p_0.$$
\end{prop}

\begin{proof}
The proof is essentially the same as that for \cite[Proposition 4.2]{BLM90} by using Theorem~ \ref{mult-stand-basis2} and Theorem \ref{monomial-basis}.
For the reader's convenience, we shall  prove it for $l = 2$.
We first assume that $A_1$ is a tridiagonal matrix.
For any $(S, T)$ satisfying Condition (\ref{star}),  $\ro(S) =\alpha$ and $\ro(T) = \alpha^J$, we set
\begin{equation*}
  h'_{S,T} = \sum_{\overset{i\in [1,n]}{i >l}} a_{ii} s_{il} - \sum_{\overset{i\in [1,n]}{i<l}} a_{ii} t_{il} - \sum_{i\in [1,n]} a_{ii}s_{ii}
  \quad {\rm and} \quad h''_{S,T}=h_{S,T}- h'_{S,T},
\end{equation*}
where  $h_{S,T}$ is defined in Theorem~ \ref{mult-stand-basis2}.
We note that $h''_{S,T}$ remains the same when $A$ is replaced by ${}_p\!A$.
For such $S,T$, we define
\begin{equation*}
\begin{split}
  G_{S,T}&(v,v') =  v^{h_{S,T}} n(S,T) \prod_{\overset{(i,j)\in \mathcal J}{i\neq j}}\begin{bmatrix}
    a_{ij}' \\ s_{ij}
  \end{bmatrix}
  \begin{bmatrix}
    a_{ij}'-s_{ij}\\ s_{-i,-j}
  \end{bmatrix}
  {v'}^{\sum_{i\in [1,n],i>l}s_{il} - \sum_{i\in [1,n], i<l} t_{il}-\sum_{i\in [1,n]}s_{ii}}\\
  &
  \prod_{i\in [-r,-1]} \prod_{\overset{1\leq k \leq s_{ii}}{1 \leq l \leq s_{-i,-i}}}
   \frac{v^{2(a_{ii}'-k+1)}v'^2-1}{v^{2k}-1}
   \frac{v^{2(a_{ii}'-s_{ii}-l+1)}v'^2-1}{v^{2l}-1}
   \prod_{\overset{i=0, -r-1}{0 \leq k \leq s_{ii}-1}} \frac{v^{2a_{ii}'-4k}v'^2-1}{v^{2k+2}-1}.
\end{split}
\end{equation*}
From Theorem~ \ref{mult-stand-basis2}, for large enough $p$, we have
\begin{equation*}
  [{}_p\!A_1]* [{}_p\!A] = \sum_{S, T} G_{S, T}(v, v^{-p}) [{}_p\! A_{S,T}].
\end{equation*}
Thus the proposition holds for the case that $A_1$ is a tridiagonal matrix.
We now assume that $A_1$ is an arbitrary matrix.
For any $p \geq 0$, there exist tridiagonal matrices $B_1, \ldots, B_s$ such that
\begin{equation}\label{eqA2}
  [{}_p\!B_1] * \cdots * [{}_p\!B_s] = [{}_p\!A_1] + \text{lower terms}.
\end{equation}
By the above proof, there exist $Z_1, \ldots, Z_m$ and $G_j(v,v')$ such that
\begin{equation}\label{eqmul}
  [{}_p\!B_1] * \cdots * [{}_p\!B_s] * [{}_p\!A_2] = \sum_{j=1}^m G_j(v, v^{-p}) [{}_p\!Z_j]
\end{equation}
for large enough $p$.
In particular, let $A_2$ be a suitable diagonal matrix.
There exist  $Z'_1, \cdots, Z'_m$ and $G'_j(v,v')$ such that
\begin{equation}\label{eqA1}
  [{}_p\!B_1] * \cdots * [{}_p\!B_s]  = \sum_{j=1}^{m'} G'_j(v, v^{-p}) [{}_p\!Z'_j]
\end{equation}
for large enough $p$.
By comparing (\ref{eqA1}) with (\ref{eqA2}), we may assume that $Z_1'=A_1$, $G_1'(v,v')=1$, and ${}_p\!Z_j < {}_p\!A$ for $j>1$ and large enough $p$.
Therefore,
\begin{equation*}
\begin{split}
   [{}_p\!A_1]* [{}_p\!A_2] &= [{}_p\!B_1] * \cdots * [{}_p\!B_s] * [{}_p\!A_2]- \sum_{j=2}^{m'} G'_j(v, v^{-p}) [{}_p\!Z'_j]* [{}_p\!A_2]
\end{split}
\end{equation*}
By (\ref{eqmul}) and induction process,  $[{}_p\!A_1]* [{}_p\!A_2]$ is of the required form.
This finishes the proof for the case $l=2$.
The general case can be completed by induction.
\end{proof}

Let $\mathcal A = \mbb Z[v,v^{-1}]$.
We introduce an $\cA$-module and a $\mbb Q(v)$-module
\[
{}_\cA\Kc_n =\cA\text{-span } \{[A] \big \vert \; A\in \widetilde{\MX}_n\},
\qquad  \Kc_n = \mbb Q(v) \otimes_{\cA}  {}_\cA\Kc_n.
\]
(Here $[A]$ are just symbols.)
By specialization at $v'=1$, we have the following.

\begin{cor} \label{cor6.2}
Retain the assumptions in Proposition \ref{prop6.1}.
 There is a unique associative $\mathcal A$-algebra structure on $\Kc_n$, without unit,  where
 the product is given by
 $$[A_1] \cdot [A_2]\cdots  [A_r] =\sum_{j=1}^m G_j(v,1)[Z_j].
 $$
 \end{cor}

By comparing Corollary ~\ref{cor6.2} with Theorem~ \ref{mult-stand-basis2}, we obtain the following multiplication formula for $\Kc_n$.

\begin{prop}\label{mult-K}
  Let $\alpha = (\alpha_i)_{i\in \mbb Z} \in \mbb N^{\mbb Z}$ such that $\alpha_i = \alpha_{i+n}$ for all $i\in \mbb Z$.
If $A, B \in \widetilde{\MX}_n$  satisfy $\co(B) = \ro(A)$ and $B - \sum_{1\leq i \leq n} \alpha_i E_{\theta}^{i, i+1}$ is diagonal, then we have
\begin{equation*}
[B] \cdot [A] = \sum_{S, T} v^{h_{S,T}} \; n(S, T) \begin{bmatrix}
  A_{S,T}\\ S
\end{bmatrix}_{\mathfrak b}  [A_{S,T}]  \in \Kc_n,
\end{equation*}
where the sum runs over all $S, T\in \Theta_n$ subject to Condition (\ref{star}),  $\ro(S) =\alpha$, $\ro(T) = \alpha^J$,
$A-T+\check{T} \in \widetilde{\Theta}_n$ \eqref{Theta:n2}, and $A_{S,T} \in {\MX}_{n}$.
\end{prop}

Given $A, B\in \widetilde{\MX}_n$, we shall denote $B \sqsubseteq A$ if
${}_pB \leq_{\text{alg}} {}_pA$ for large enough $p\in \mbb N$, ${\rm co}(B)={\rm co}(A)$, and ${\rm ro}(B)={\rm ro}(A)$.
We write $B \sqsubset A$ if $B \sqsubseteq A $ and $B\neq A$.
By using Proposition \ref{mult-K} and a similar argument as for Theorem ~\ref{monomial-basis}, we have the following.

\begin{prop}
 \label{monomial-basis-K}
For any $A \in \widetilde{\MX}_{n}$, there exist matrices $B(i)\in \widetilde{\MX}_{n}$
with $B(i) - \sum_{1\leq j \leq n} \alpha_{ij} E_{\theta}^{j, j+1}$ being diagonal for suitable scalars $\alpha_{ij}$
such that
\begin{align*}
m_A' &=  \prod_i [B(i)] \in  [A] + \sum_{A' \sqsubset A} \cA [A'];
 \\
m_A &=  \prod_i \{B(i)\} \in  [A] + \sum_{A' \sqsubset A} \cA [A'].
\end{align*}
\end{prop}

We have the following stabilization property for the bar operator from ~\cite[Proposition 9.2.7]{FLLLWa}.

\begin{lem}\label{lem1}
 For any $A\in \tilde{\Xi}_{n}$, there exist $T_1, \ldots, T_m \in \tilde{\Xi}_{n}$, $H_i(v,v') \in \mathcal R$ and $p_0 \in \mathbb N$ such that
 \begin{equation*}
   \overline{[{}_p\! A]} = \sum_{i=1}^m H_i(v, v^{-p}) [{}_p\! T_i], \ \forall p\geq p_0.
 \end{equation*}
\end{lem}

By specializing at $v'=1$, we define a $\mathbb Q$-linear map $\bar{\phantom{x}}: \Kc_n \rightarrow \Kc_n$ by letting
$\overline{v^j [A]} = v^{-j} \sum_{i=1}^m H_i(v,1) [T_i].$
By Lemma~ \ref{lem1},  $\ \bar{\phantom{x}}\ $ is a ring homomorphism whose square is the identity.
From  Lemma~ \ref{lem1}, we also have
$\overline{[A]}  = [A] +{\rm lower\ terms},\ \forall A\in \tilde{\Xi}_{n}.$
By using Proposition \ref{monomial-basis-K}, we have
$\overline{[A]} \in [A] + \sum_{C \sqsubset A} \cA [C].$

By a  similar argument as for \cite[Proposition~ 4.7]{BLM90}, we have the following.

\begin{prop}\label{caonical-basis}
For any $A\in \widetilde{\MX}_{n}$, there exists a unique element $\{A\}$ in $\Kc_n$ such that
$$\overline{\{A\}}=\{A\},\quad \{A\}=[A]+\sum_{A'\sqsubset A, A'\neq A}\pi_{A', A}[A'],
\quad \pi_{A', A} \in v^{-1} \mbb Z [v^{-1}].$$
\end{prop}

By Propositions \ref{monomial-basis-K} and \ref{caonical-basis}, we have the following.

\begin{thm}
 \label{thm:CBgl}
The algebra $\Kc_n$ possesses 
a monomial basis $\{m_{A} \big \vert  A\in \widetilde{\MX}_{n}\}$
and a canonical basis $\{ \{ A\} \big \vert A\in \widetilde{\MX}_{n}\}$.
\end{thm}

\subsection{Homomorphism from $\Kc_n$ to $\Sj$}

By comparing the multiplication formulas in Theorem~\ref{mult-stand-basis2}
and Proposition \ref{mult-K},
we  can establish a further connection between $\Kc_n$ and $\Sj$.
Note that $\MX_{n,d} \subset \widetilde{\MX}_{n,d}$.
Let $\Psi: \Kc_n \rightarrow \Sj$ be an $\mbb Q (v)$-linear map
defined by
\[
\Psi([A]) =
\begin{cases}
[A], & \mbox{if} \ A\in \MX_{n,d},\\
0, & \mbox{otherwise}.
\end{cases}
\]

\begin{prop}
The map $\Psi$   is a surjective algebra homomorphism.
\end{prop}

\begin{proof}
The subjectivity follows by definition of $\Psi$.

We show that $\Psi$ is an algebra homomorphism.
The proof is similar to that for \cite{DF13} 
or that for \cite[Lemma~ A.20]{BKLW14}.
By Proposition~\ref{monomial-basis-K}, it is enough to show that
\begin{equation}\label{eq1}
  \Psi( [B] \cdot [A]) = \Psi([B]) \Psi([A])
\end{equation}
for all tridiagonal matrices $B \in \widetilde{\Xi}_{n}$.

It follows by comparing multiplication formulas in Theorem~\ref{mult-stand-basis2}
and Proposition \ref{mult-K}, Equation (\ref{eq1}) holds for $B, A \in \Xi_{n,d}$ with $B$ tridiagonal.

If $A \not \in \Xi_{n,d}$, then there exists $i \in [1,n]$ such that $a_{ii} < 0$.
By Condition \eqref{star} in Lemma~\ref{lem:star}, for any $T, S \in \Theta_{n}$
such that $A_{S,T} =(a_{ij}' ) \in \Xi_n$, we have
\[a_{ii}' - s_{ii} = a_{ii} + s_{-i,-j} - (t_{ij} + t_{-i,-j}) < s_{-i,-j},\]
and hence, $\begin{bmatrix}
  A_{S,T}\\
  S
\end{bmatrix}_{\C}=0.$
This implies $\Psi([B] \cdot [A])=0 = \Psi([B]) \Psi([A])$.

If $B \not \in \Xi_{n,d}$, then there exists $i \in [1,n]$ such that $b_{ii} < 0$.
Thanks to ${\rm co}(B) = {\rm ro}(A)$, we have $\sum_j a_{ij} = b_{ii} + \alpha_{i-1} + \alpha_{n-1-i}$.
By Condition \eqref{star} in Lemma~\ref{lem:star} again, for any $T, S \in \Theta_{n}$
with $A_{S,T} =(a_{ij}' ) \in \Xi_n$, we have
\[\sum_j (a_{ij}' - s_{ij}) =\sum_j a_{ij} + \sum_j s_{-i,-j} - \sum_j t_{ij} + \sum_j t_{-i,-j})
= b_{ii} + \alpha_{-i} < \sum_j s_{-i,-j}.\]
There exists $k \in \mathbb Z$ such that $a_{ik}' - s_{ik} < s_{-i,-k}$.
Therefore, $\begin{bmatrix}
  A_{S,T}\\
  S
\end{bmatrix}_{\C}=0$ for any $T, S \in \mathfrak S_{n,n}$.
This implies $\Psi([B] \cdot [A])=0 = \Psi([B]) \Psi([A])$.

The proposition is proved.
\end{proof}



\section{The algebras $\Kcji$,  $\Kcij$ and $\Kcii$}
 \label{chap:K234}

In this Section, we adapt the constructions of the monomial basis  of $\Sj$ in Section~\ref{secstandbasis} for
the remaining 3 variants of convolution algebras:  ${\mbf S}^{\ji}_{\nn, d}$, ${\mbf S}^{\ij}_{\nn, d}$, and ${\mbf S}^{\ii}_{\mm, d}$.
This then allows us to establish the stabilization properties and construct the corresponding limit algebras $\Kcji$, $\Kcij$, and $\Kcii$, respectively.
Monomial and canonical bases for $\Kcji$, $\Kcij$, and $\Kcii$ are also constructed.
We further establish an isomorphism $\Kcji \cong \Kcij$ with compatible monomial, standard and canonical bases.

\subsection{Monomial and canonical bases for ${\mbf S}^{\ji}_{\nn,d}$ and $\Kcji$}
\label{sec:Sji}

Recall $n=2r+2$ (for $r \ge 1$) is even, and $\nn =n-1 =2r+1$.

Recall the subset $\MX^{\ji}_{\nn,d} \subset \MX_{n,d}$ from \eqref{Mjid}
and the subalgebra ${\mbf S}^{\ji}_{\nn,d} = \bj_r \Sj \bj_r$ of $\Sj$ from Section \ref{Sji}.
Note the tridiagonal matrices $B$ with nonzero $(r+1, r)$th or $(r,r+1)$th entry are not in $\MX_{\nn,d}^{\jmath\imath}$,
and so a generating set for the algebra ${\mbf S}^{\ji}_{\nn,d}$ does not naively come from that for $\Sj$.
Recall  the matrices $E^{ij}$ and $E^{ij}_\theta$
from Section \ref{Para}.

\begin{thm}
  \label{monomial-basis-ji}
Let $A \in \MX_{\nn,d}^{\jmath \imath}$. There exist matrices $B(i)\in  \Xi_{\nn,d}^{\jmath \imath}$ $\,(i\ge 0)$
with  
$B(i) -\sum_{j \in [1,n]\backslash\{ r, r+1\} }  c_{i,j} E^{j,j+1}_{\theta} - c_{i,r+1} E^{r, r+2}_{\theta}$ being diagonal (for some
scalars $c_{i,j}$) such that
\begin{align*}
{}'\texttt{M}_A^{\, \ji} &:=  \prod_{i\ge 0}^{\longleftarrow} [B(i)] = [A] + {\rm lower\ terms} \in {\mbf S}^{\ji}_{\nn,d},
 \\
\texttt{M}_A^{\, \ji} &:=  \prod_{i\ge 0}^{\longleftarrow} \{B(i)\} = [A] + {\rm lower\ terms} \in {\mbf S}^{\ji}_{\nn,d}.
\end{align*}
\end{thm}

\begin{proof}
The construction below is inspired by a  similar construction in the proof of \cite[Theorem 6.3.1]{FL14}, see also ~\cite{BKLW14, BLW14}.
We shall work in the framework of the larger algebra $\Sj$ instead of ${\mbf S}^{\ji}_{\nn,d}$.

For $A = (a_{ij}) \in \MX_{\nn,d}^{\jmath \imath} \subset \MX_{n,d}$,
there exist matrices $B'(i)$ (as constructed in the proof of  Theorem~ \ref{monomial-basis} without  the prime notation)
such that $B'(i) - \sum_{1\leq j \leq n} \alpha_{ij} E_{\theta}^{j, j+1}$ is diagonal and
\begin{equation}\label{eq4}
  \prod_{i\geq 0}^{\longleftarrow}  [B'(i)] = [A] + \text{lower terms} \in \Sj.
\end{equation}
In this proof we use freely the setup and notations in the proof of Theorem~ \ref{monomial-basis}.
We emphasize  that $B'(i)$ for $i\ge 1$ are not necessarily in $\MX_{\nn,d}^{\jmath \imath}$.

Recall from the proof of Theorem~ \ref{monomial-basis} that $m \in \mbb N$ is fixed such that $a_{ij}=0$ for all $|j-i| >m$
and $\alpha_{ij}= \sum_{k =  j+i-m}^{j} a_{k,j+i}$ (for $i\ge 1$).
For $i\ge 2$, thanks to $a_{r+1,r+i}=0$ we have
\begin{equation}
  \label{eqalpha}
\alpha_{i-1, r+1} = \sum_{k=r+i-m}^{r+1}a_{k,r+i} = \sum_{k=r+i-m}^{r}a_{k,r+i} = \alpha_{i,r}.
\end{equation}
Moreover, $\alpha_{i, r+1-i} =0$ for $i\ge 1$.
For $i\ge 1$, we denote by $C(i)$ (respectively, $D(i)$) the matrix such that $C(i) - \sum_{r+1 \leq j \leq 3r+2-i} \alpha_{ij} E^{j,j+1}_{\theta}$
(respectively, $D(i) -  \sum_{r+1-i \leq j \leq r} \alpha_{ij} E^{j,j+1}_{\theta}$) is a diagonal matrix
and ${\rm ro}(C(i)) = {\rm ro}(B'(i))$ (respectively, ${\rm co}(D(i)) = {\rm co}(B'(i))$).
It follows by Proposition~\ref{leading} that, for $i \ge 1$,
\begin{equation}\label{eq5}
  [C(i)] * [D(i)] = [B'(i)] + {\rm lower\ terms}.
\end{equation}
We note that the decomposition in (\ref{eq5}) is highly dependent on the condition $\alpha_{i, r+1-i} =0$,
and there always exists such a decomposition whenever there exists $\alpha_{ij} = 0$.

We set $C(0) :=B'(0)$. By Proposition~\ref{leading} and (\ref{eqalpha}), we have, for $i\ge 0$,
\begin{equation}
  \label{eq6}
  [D(i+1)] * [C(i)] = [B(i)] + {\rm lower\  terms},
\end{equation}
where $B(i) \in \MX_{\nn,d}^{\jmath \imath}$ satisfies that $B(i) - \alpha_{i,r+1} E^{r, r+2}_{\theta}$ is a tridiagonal matrix.
(It can be shown by Theorem~\ref{mult-stand-basis2}  that $[D(i+1)] * [C(i)] \in \MX_{\nn,d}^{\jmath \imath}$; but we do not need this stronger fact.)
Therefore, it follows by (\ref{eq5}) and (\ref{eq6}) that
\begin{equation}\label{eq4b}
  \prod_{i\geq 0}^{\longleftarrow}  [B'(i)] = \prod_{i\geq 0}^{\longleftarrow}  [B(i)] + L \in \Sj,
\end{equation}
where $L$ is the product of lower terms than $B(i)$, which is lower than the leading term in $\prod_{i\geq 0}  [B(i)]$.
The theorem follows now by comparing (\ref{eq4}) and (\ref{eq4b}).
\end{proof}

\begin{cor}
The set $\{ \texttt{M}_A^{\, \ji} \vert A \in \MX_{\nn,d}^{\ji} \}$
(resp., $\{ '\texttt{M}_A^{\, \ji} \vert A \in \MX_{\nn,d}^{\ji} \}$)
forms a basis of ${\mbf S}^{\ji}_{\nn,d}$.
\end{cor}
We call the basis $\{ \texttt{M}_A^{\jmath \imath} \vert A \in \MX_{\nn,d}^{\ji} \}$ the monomial basis of ${\mbf S}^{\ji}_{\nn,d}$.

Recalling $I$ is the identity matrix, we set
\begin{equation*}
\begin{split}
I^{\jmath \imath} = I - E_{r+1, r+1},
\qquad
 & {}_pA^{\jmath \imath} = A + p I^{\jmath \imath}\; (\forall p\in 2\mbb N).
\end{split}
\end{equation*}
We set ${\MX}_{\nn}^{\ji}:=\sqcup_{d} {\MX}_{\nn,d}^{\ji}$
We extend ${\MX}_{\nn}^{\ji}$ to a larger set
$\widetilde{\MX}_{\nn}^{\ji}$ by  requiring  the diagonal entries to be  in $\mbb Z$, instead of $\mbb N$ .
For any given $A \in \widetilde{\MX}_{\nn}^{\ji}$, we have ${}_pA^{\ji} \in {\MX}_{\nn}^{\ji}$ for $p\gg 0$.
The following stabilization is slightly different from that for $\Kc_n$ and is  similar to that in \cite{FL14}.
For the reader's convenience, we present the construction here.

\begin{prop}
\label{prop7.1}
  Suppose that $A_1, A_2,\ldots, A_l  \in \widetilde{\MX}_{\nn}^{\ji} \ (l \geq 2)$ are
such that ${\rm co}(A_i)={\rm ro}(A_{i+1})$ for all $i$.
Then there exist $Z_1, \ldots, Z_m\in \widetilde{\MX}_{\nn}^{\ji}$, $G_j(v,v')\in \mathcal R$ and $p_0\in \mbb N$ such that
$$[{}_p A_1^{\jmath \imath}] * [{}_pA_2^{\jmath \imath}] * \cdots *[{}_p A_l^{\jmath \imath}]=\sum_{j=1}^mG_j(v,v^{-p})[{}_p Z_j],\quad
\forall p \in 2 \mbb N, p\geq p_0.$$
\end{prop}

\begin{proof}
It suffices to prove the proposition for $l=2$.
  Let us first assume that $A_1$ is a tridiagonal matrix in $\widetilde{\MX}_{\nn,d}^{\jmath \imath}$.
  For any $(S, T)$ satisfying Condition (\ref{star}),  $\ro(S) =\alpha$ and $\ro(T) = \alpha^J$, we define
  \begin{equation*}
\begin{split}
  & G_{S,T}(v,v') =  v^{h_{S,T}} n(S,T) \prod_{\overset{(i,j)\in \mathcal J}{i\neq j}}\begin{bmatrix}
    a_{ij}' \\ s_{ij}
  \end{bmatrix}
  \begin{bmatrix}
    a_{ij}'-s_{ij}\\ s_{-i,-j}
  \end{bmatrix}
   \prod_{\overset{i=-r-1}{0 \leq k \leq s_{ii}-1}} \frac{v^{2a_{ii}'-4k}-1}{v^{2k+2}-1}
  v'^{\gamma(S,T)}\\
  &
  \cdot \prod_{i\in [-r,-1]} \prod_{\overset{1\leq k \leq s_{ii}}{1 \leq l \leq s_{-i,-i}}}
   \frac{v^{2(a_{ii}'-k+1)}v'^2-1}{v^{2k}-1}
   \frac{v^{2(a_{ii}'-s_{ii}-l+1)}v'^2-1}{v^{2l}-1}
   \prod_{\overset{i=0}{0 \leq k \leq s_{ii}-1}} \frac{v^{2a_{ii}'-4k}v'^2-1}{v^{2k+2}-1},
\end{split}
\end{equation*}
where $\gamma(S,T) = \sum_{\overset{i\in [1,n], i>l}{ i\neq r+1}}s_{il} - \sum_{\overset{i\in [1,n], i<l}{i\neq r+1}} t_{il}-\sum_{\overset{i\in [1,n]}{ i\neq r+1}}s_{ii}$.
We note that the definition of $G_{S,T}(v,v')$ is slightly different from the one in the proof of Proposition \ref{prop6.1}.
The difference is that the index $i$ can not be equal to $ r+1$ in $\gamma(S,T)$ by the definition of ${}_p\!A_1^{\jmath \imath}$.

From Theorem~ \ref{mult-stand-basis2}, for large enough $p$, we have
\begin{equation*}
  [{}_p\!A_1^{\jmath \imath}]* [{}_p\!A^{\jmath \imath}] = \sum_{S, T} G_{S, T}(v, v^{-p}) [{}_p\! A_{S,T}^{\jmath \imath}].
\end{equation*}
Thus the proposition holds for the case that $A_1$ is a tridiagonal matrix.
For an arbitrary matrix $A_1$, the argument is exactly the same as the one in the proof of Proposition \ref{prop6.1}.
It is clear that $Z_j\in \widetilde{\MX}_{\nn,d}^{\jmath \imath}$ since ${\rm ro}({}_p\! Z_j) = {\rm ro}({}_p\!A_1)$ and
${\rm co}({}_p\! Z_j) = {\rm co}({}_p\!A_l)$.
This finishes the proof.
\end{proof}

We introduce an $\cA$-module and $\mbb Q(v)$-module
$$
{}_{\cA} \Kcji = {\rm span}_{\mathcal A}\{ [A]| A \in \widetilde{\MX}_{\nn}^{\ji} \},
\qquad  \Kcji =\mbb Q(v) \otimes_{\cA} {}_{\cA} \Kcji.
$$
(Here $[A]$ is just a symbol.)
By specialization  at $v'=1$ in Proposition~\ref{prop7.1}, we can define an algebra structure on
${}_{\cA} \Kcji$ and $\Kcji$.

\begin{cor} \label{cor7.2}
Retain the assumption from Proposition~ \ref{prop7.1}.
 There is a unique associative $\mbb Q(v)$-algebra structure on $\Kcji$, without unit,  where
 the product is given by
 $$[A_1] \cdot [A_2]\cdots[A_r] =\sum_{j=1}^m G_j(v,1)[Z_j].
 $$
 \end{cor}

The following  theorem for $\Kcji$ is a counterpart of Theorem~\ref{thm:CBgl} for the algebra $\Kc_n$.

\begin{thm} \label{Kji}
\begin{enumerate}
\item
The algebra $\Kcji$ is generated by $[B]$, where
$B\in \widetilde{\MX}^{\ji}_{\nn}$ are such that
$B -\sum_{j \in [1,n]\backslash\{r, r+1\} }  \alpha_j E^{j,j+1}_{\theta} - \alpha_{r+1} E^{r, r+2}_{\theta}$ is diagonal for suitable scalars $\alpha_j$.

\item
The algebra $\Kcji$ possess  
a monomial basis $\{ \texttt{M}_A^{\ji} | A\in \widetilde{\MX}^{\ji}_{\nn}\}$,  and a canonical basis $\{ \{ A\} | A\in \widetilde{\MX}^{\ji}_{\nn}\}$.

\item
There exists a surjective algebra homomorphism $\Psi^{\jmath \imath}: \Kcji \rightarrow {\mbf S}^{\ji}_{\nn,d}$ such that
$\Psi^{\jmath \imath} ([A]) = [A]$ if $A\in {\MX}^{\ji}_{\nn,d}$ and $0$ otherwise.
\end{enumerate}
\end{thm}

\begin{proof}
Parts (1) and (2) follow by the construction of the monomial basis  for  ${\mbf S}^{\ji}_{\nn,d}$ (Theorem~\ref{monomial-basis-ji})
and the stabilization procedure (Proposition~ \ref{prop7.1}). Part~ (3) follows formally by the existence of the monomial basis
with a desirable triangularity property under the bar operator (compare Proposition~\ref{caonical-basis}).
\end{proof}

\subsection{The algebras of type $\ijw$}  

The $\ijw$-type is completely parallel to the $\jiw$-type treated above, and so we shall be brief.
Recall the subset $\MX^{\ij}_{\nn,d} \subset \MX_{n,d}$ from \eqref{Mijd}
and the subalgebra ${\mbf S}^{\ij}_{\nn,d} = \bj_0 \Sj \bj_0$ of $\Sj$ from Section \ref{Sji}.
Recalling $I$ is the identity matrix, we set
\begin{equation}
\label{Iij}
I^{\ij} = I - E_{00},
\qquad
  {}_pA^{\ij} = A + p I^{\ij}\; (\forall p\in 2\mbb N).
\end{equation}
As an $\ijw$-counterpart of Theorem~\ref{monomial-basis-ji}, the algebra ${\mbf S}^{\ij}_{\nn,d}$ is generated by
$[B]$,  for all $B \in \MX^{\ij}_{\nn,d}$ such that
$B-\sum_{j \in [1,n]\backslash\{0,n-1\}  } \alpha_{j} E^{j,j+1}_{\theta} - \alpha_{0} E^{-1,1}_{\theta}$ is diagonal
for some suitable scalars $\alpha_{j}$.
Similarly, we can construct a monomial basis $\{ \texttt{M}_A^{\ij} \vert A \in \MX^{\ij}_{\nn,d} \}$ for ${\mbf S}^{\ij}_{\nn,d}$.

Define $\widetilde{\MX}_{\nn}^{\ij}$ in a similar manner as $\widetilde{\MX}_{\nn}^{\ji}$.
Let ${}_{\cA}\Kcij = {\rm span}_{\mathcal A}\{ [A]| A \in \widetilde{\MX}_{\nn}^{\imath \jmath}\}$, and $\Kcij =\mbb Q(v) \otimes_{\cA} {}_{\cA}\Kcij$.
Similar to Proposition~\ref{prop7.1}, we can establish a stabilization property  for the algebras ${\mbf S}^{\ij}_{\nn,d}$ as $d \mapsto \infty$ using \eqref{Iij},
and then we can use the stabilization procedure to construct an algebra structure on $\Kcij$
(similar to Corollary~\ref{cor7.2}). As a $\ijw$-counterpart of Theorem~\ref{Kji}, $\Kcij$ admits a monomial basis
$\{ \texttt{M}_A^{\ij} | A\in \widetilde{\MX}^{\ij}\}$  and a canonical basis $\{ \{ A\} | A\in \widetilde{\MX}^{\ij}\}$.
There exists a surjective algebra homomorphism $\Psi^{\ij}: \Kcij \rightarrow {\mbf S}^{\ij}_{\nn,d}$ such that
$\Psi^{\ij} ([A]) = [A]$ if $A\in {\MX}^{\ij}_{\nn,d}$ and $0$ otherwise.

\subsection{Isomorphisms between the types $\jiw$ and $\ijw$}

The similarity between the $\jiw$-type and $\ijw$-type which we have seen above is not accidental and we shall establish various isomorphisms below.
Let us start at the levels of convolution algebras.

To a matrix $A \in \MX_{n,d}$, we associate  $^\tau A=(^\tau a_{ij})$ and $^{\tau} a_{ij} = a_{i+r+1, j+r+1}$.
Clearly, sending $A \mapsto {}^{\tau}A$ defines a permutation of order $2$ on $\MX_{n,d}$.

\begin{prop}
 \label{prop:involution}
The assignment  $[A] \mapsto [\ \! ^{\tau} A]$ defines an involution  $\tau_d$ on the algebra $\Sj$.
\end{prop}

\begin{proof}
We must show that $\tau_d ([B] * [A])
= [^\tau B ] * [^\tau A]$ for all $B, A \in \MX_d$.
By using the monomial basis $\texttt{M}_A$ of $\Sj$ and by induction on $B$ with respect to the partial order $\leq_{\text{alg}}$,
we only need to prove the equation for $B$ such that $B- \sum_{1 \leq i \leq n} \alpha_i E^{i, i+1}_{\theta}$ is diagonal for some $\alpha_i$.

Note that $d_A = d_{\  \! ^\tau \! A}$ for all $A$.
By the multiplication formula  (Theorem~\ref{mult-stand-basis2}),  it suffices to check that
\begin{align*}
n(\ \!  ^{\tau} S, \ \! ^{\tau} T) = n(S, T), \quad
\begin{bmatrix} \ \! ^{\tau} A_{S, T} \\ ^\tau S \end{bmatrix}_{\C} = \begin{bmatrix} A_{S, T} \\ S \end{bmatrix}_{\C}, \quad
\xi^{\mathfrak b}_{^{\tau} A, \ \! ^\tau S, \ \! ^\tau T} = \xi^{\mathfrak b}_{A, S, T},
\end{align*}
where $S, T$ are from the multiplication formula. Indeed, we have
\begin{align}
\begin{split}
n(\ \!  ^{\tau} S, \ \! ^{\tau} T)
&= \prod_{0 \leq i \leq r} n(^{\tau} S_i,  \ ^{\tau} T_i, \ ^{\tau} S^J_{-i -1}) = \prod_{0 \leq i \leq r} n(S_{i+ r+1}, T_{i+r+1}, S^J_{r - i}) \\
& = \prod_{0 \leq i \leq r} n(S_{r - i}, T_{r-i}, S^J_{i + r +1}) = n(S, T),
\end{split}
\end{align}
where we have used Corollary~ \ref{weak-n-dual} and \eqref{star} on the third equality.

Observe that
\begin{align*}
\begin{bmatrix} a'_{i+r+1, j+r+1} \\ s_{i+r+1, j+r+1} \end{bmatrix}
&\begin{bmatrix} a'_{i+r+1, j+r+1} - s_{i+r+1, j+r+1} \\ s_{-i + r+1, -j + r + 1} \end{bmatrix}
\\
=
&\begin{bmatrix} a'_{-i -r -1, -j -r -1} \\ s_{-i -r -1, -j -r -1} \end{bmatrix}
\begin{bmatrix} a'_{ -i -r -1, -j -r -1} - s_{-i -r -1, -j -r -1} \\ s_{i + r +1, j + r +1}\end{bmatrix}.
\end{align*}
Observe also that sending $(i, j) \mapsto (-i -r-1, -j -r -1)$ defines a permutation on the index set $\mathcal J$ in
$\begin{bmatrix} A_{S, T} \\ S \end{bmatrix}_{\C}$  (see Theorem~\ref{mult-stand-basis2}).
It follows from these observations  that
$$
\begin{bmatrix} \  \! ^{\tau} A_{S, T} \\ ^\tau S \end{bmatrix}_{\C} = \begin{bmatrix} A_{S, T} \\ S \end{bmatrix}_{\C}.
$$
The equality $\xi^{\mathfrak b}_{^{\tau} A, \ \! ^\tau S, \ \! ^\tau T} = \xi^{\mathfrak b}_{A, S, T}$ follows by using the symmetries on $A$, $S$ and $T$, and
the proposition is thus proved.
\end{proof}

\begin{prop}
\label{tau:CBSj}
\begin{enumerate}
\item
The automorphism  $\tau_d: \Sj \to \Sj$ commutes with the bar involution.

\item
The involution $\tau_d$ preserves the canonical basis, i.e.,  $\tau_d \{ A\} = \{ ^\tau A\}$, for $A\in {\MX}_{n,d}$.
\end{enumerate}
\end{prop}

\begin{proof}
Let us go back to the geometric setting of Section~\ref{rec}.
Let $J$ be the anti-diagonal matrix of rank $d$.
Let us fix a basis for the vector space $V$
so that the associated matrix of the symplectic form is
$
\begin{pmatrix}
0 & J\\
-J & 0
\end{pmatrix}.
$
Let
$
\sigma =
\begin{pmatrix}
0 & \varepsilon J \\
J & 0
\end{pmatrix}.
$
For any $u, v \in V_F$, we have
\[
(\sigma u, \sigma v)_{V_F} = - \varepsilon ( u, v)_{V_F}.
\]
Hence $\sigma $ is an element in the group $\mrm{GSp} (V_F)$ of symplectic similitude with respect to the form $(, )_{V_F}$.
Note that $\mrm{GSp}(V_F)$ acts on the set of symplectic lattices, and hence on $\X^{\C}_{n, d}$.

For any $L\in \X^{\fc}_{n,d}$, we set $^{\tau}L = (^{\tau}L_i)_{i\in \mbb Z}$ where $^{\tau}L_i = \sigma( L_{i+r+1})$ for all $i\in \mbb Z$.
We have
\[
(^{\tau} L_i, \ ^{\tau} L_{-i -1} )
= - \varepsilon ( L_{i+r+1}, L_{-i + r})
= \varepsilon ( L_{i+r+1}, \varepsilon^{-1} L^{\#}_{i+r+1}) \in \mathbb F_q [[ \varepsilon]].
\]
Thus we have $(^{\tau}L_i)^{\#} = \ ^{\tau} L_{-i -1}$, i.e., $^{\tau}L\in \X^{\fc}_{n,d}$. Hence we have defined a bijection
$\tau : \X^{\fc}_{n,d} \longrightarrow \X^{\fc}_{n,d}.$
Over the algebraic closure $\overline{\mbb F}_q$,
$\tau$ is an isomorphism of ind-varieties and the push forward $\tau_*=\tau_!$ commutes with the Verdier duality.
Since the automorphism $\tau_d$ is the decategorified version of $\tau_*$ and the bar involution is the decategorified version of  the Verdier duality,
Part~(1) follows.

By Proposition~\ref{prop:involution}, we have $\tau_d([A]) = [{}^{\tau} A]$. As the bijection on $\MX_{n,d}$ sending $A \mapsto {}^{\tau}A$
preserves the Bruhat ordering, it follows by definition of the canonical basis that $\tau_d \{ A\}$ satisfies the characterization property of
${}^\tau \{ A\}$, and hence $\tau_d \{ A\} = \{ ^\tau A\}$.
\end{proof}

Recall Lusztig's subalgebra $\bU^{\ji}_{\nn,d}=\langle \tji_r, \eji_i, \fji_i\rangle_{0\leq i\leq r-1} $ and
$\bU^{\ij}_{\nn,d}=\langle \tij_0, \eij_i, \fij_i\rangle_{0\leq i\leq r-1}$ of $\Sji$ and $\Sij$, respectively,  from ~\cite[Sections 7.1 and 8.1]{FLLLWa}.

\begin{thm}
\label{thm:tau}
The involution $\tau_d$ on the algebra $\Sj$ restricts to algebra isomorphisms $\tau_d: {\mbf S}^{\ji}_{\nn,d} \stackrel{\cong}{\longrightarrow} {\mbf S}^{\ij}_{\nn,d}$
and $\tau_d: \bU^{\ji}_{\nn,d} \stackrel{\cong}{\longrightarrow} \bU^{\ij}_{\nn,d}$ which preserve the canonical bases.
Moreover,  the isomorphism $\tau_d: \bU^{\ji}_{\nn,d}\rightarrow \bU^{\ij}_{\nn,d}$
gives rise to the following correspondence of generators:
$\tji_r \mapsto \tij_0$, $\eji_i \mapsto \fij_{r-i}$, $\fji_i \mapsto \eij_{r-i}$ and $\kji^{\pm 1}_i \mapsto \kij^{\mp 1}_{r-i} $ for all $i\in [0, r-1]$.
\end{thm}

\begin{proof}
It follows by Proposition~\ref{prop:involution} that the restriction gives rise to an algebra isomorphism
$\tau_d: {\mbf S}^{\ji}_{\nn,d} \stackrel{\cong}{\longrightarrow} {\mbf S}^{\ij}_{\nn,d}$. A direct verification by definition shows that
$\tau_d$ sends the generators of $\bU^{\ji}_{\nn,d}$ to
the corresponding generators of $\bU^{\ij}_{\nn,d}$ as stated in the proposition, and
hence we have obtained an algebra isomorphism $\tau_d: \bU^{\ji}_{\nn,d}\rightarrow \bU^{\ij}_{\nn,d}$.

The canonical bases are compatible with the inclusions ${\mbf S}^{\ji}_{\nn,d} \subset \Sj$
and ${\mbf S}^{\ij}_{\nn,d} \subset \Sj$ by the geometric definitions of these algebras.
By ~\cite[Propositions 7.3.4 and 8.2.4]{FLLLWa}, the canonical bases are also compatible with the inclusions
$\bU^{\ji}_{\nn,d} \subset {\mbf S}^{\ji}_{\nn,d}$ and $\bU^{\ij}_{\nn,d} \subset {\mbf S}^{\ij}_{\nn,d}$.
Hence the restrictions of $\tau_d$ still preserve the canonical bases by
Proposition~\ref{tau:CBSj}(2).
\end{proof}

It is straightforward to define an involution $\tau^{\fa}_d$ on the affine type $A$ convolution algebra ${\mbf S}_{n,d}$, similar to the involution $\tau_d$ on $\Sj$.
Then we have the following commutative diagram:
\begin{equation}
  \label{CD:autom}
\begin{CD}
\bU^{\ji}_{\nn,d} @>>> {\mbf S}^{\ji}_{\nn,d} @>>> \Sj @>\jmath_d >> {\mbf S}_{n,d}  \\
@V\tau_d VV @V\tau_d VV @V\tau_d VV  @V\tau^{\fa}_d VV  \\
\bU^{\ij}_{\nn,d}  @>>> {\mbf S}^{\ij}_{\nn,d}  @>>> \Sj @>\jmath_d >>{\mbf S}_{n,d},
\end{CD}
\end{equation}
where $\jmath_d$ is the coideal imbedding defined in Equation (5.3.9) from ~\cite{FLLLWa}.
Note that the commutativity of the rightmost square can be shown by the trick of imbedding
$\Sj$ to a higher rank Lusztig algebra and is then reduced to checking the compatibility on Chevalley generators.

The bijective map $\tau: \Xi_{n,d} \rightarrow \Xi_{n,d}$ sending $A \mapsto {}^{\tau}A$
 can be extended to $\tilde{\tau}: \tilde{\Xi}_{n} \rightarrow \tilde{\Xi}_{n}$, $A \mapsto {}^{\tau} A$.
By restriction we obtain a bijection $\tilde{\tau}: \tilde{\Xi}^{\jmath\imath}_{n} \rightarrow \tilde{\Xi}^{\imath\jmath}_{n}$.
Define a $\mbb Q(v)$-linear map
\[
\tau^{\jmath\imath}: \Kcji \longrightarrow \Kcij,
\quad [A] \mapsto [{}^{\tilde{\tau}} A].
\]
Theorem~\ref{thm:tau} on the algebra isomorphism $\tau_d: {\mbf S}^{\ji}_{\nn,d} \stackrel{\cong}{\longrightarrow} {\mbf S}^{\ij}_{\nn,d}$
and the stabilization procedure in Proposition~\ref{prop7.1} for $\Kcji$ (and its $\ij$-counterpart)
quickly lead to the following.

\begin{thm}
 \label{thm:tauK}
The map $\tau^{\jmath\imath}: \Kcji \longrightarrow \Kcij$ is an isomorphism of algebras, which preserves the monomial and canonical bases.
\end{thm}

\subsection{Monomial and canonical bases for algebras of type $\iiw$}  

Recall $n=2r+2$, $\nn=n-1$, and $\mm =n-2 =2r$, for $r\ge 1$.
Recall the subset $\MX^{\ii}_{\mm,d} = \MX^{\ji}_{\nn,d} \cap \MX^{\ij}_{\nn,d} \subset \MX^{\ji}_{n,d}$
from \eqref{Miid} and  the subalgebra ${\mbf S}^{\ii}_{\mm,d}  = {\mbf S}^{\ji}_{\nn,d} \cap {\mbf S}^{\ij}_{\nn,d}$
from Section ~\ref{rec}.
By a similar argument as that for Theorem \ref{monomial-basis-ji}, we have the following.

\begin{prop}
  \label{monomial-basis-ii}
  Let $A \in \MX_{\mm,d}^{\ii}$. There exist matrices $B(i)\in  \MX_{\mm,d}^{\ii}$ $\,(i\ge 0)$
with  $B(i) -\sum_{j \in [1,n]\backslash\{r, r+1,0,n-1\} }  c_{i,j} E^{j,j+1}_{\theta} - c_{i,0} E^{-1,1}_{\theta} - c_{i,r+1} E^{r, r+2}_{\theta}$ being diagonal (for some
suitable scalars $c_{i,j}$) such that
\begin{align*}
{}'\texttt{M}_A^{\,\ii} &:=  \prod_{i\ge 0}^{\longleftarrow} [B(i)] = [A] + {\rm lower\ terms} \in {\mbf S}^{\ii}_{\mm,d},
 \\
\texttt{M}_A^{\,\ii} &:=  \prod_{i\ge 0}^{\longleftarrow} \{B(i)\} = [A] + {\rm lower\ terms} \in {\mbf S}^{\ii}_{\mm,d}.
\end{align*}
\end{prop}

Let
\begin{equation}
 \label{pii}
I^{\imath \imath} = I - E_{00} -E_{r+1,r+1},
\quad
{}_pA^{\imath \imath} = A + p I^{\imath \imath} \; (\forall p\in 2 \mbb N).
\end{equation}
Let $\MX_{\mm}^{\ii} = \sqcup_{d\in \mbb N} \MX^{\ii}_{\mm, d}$ and extend it to a larger set
 $\widetilde{\MX}_{\mm}^{\ii}$ by requiring the diagonal entries in $\mbb Z$ instead of $\mbb N$, except at $(0, 0)$, $(r+1, r+1)$ mod $n$.
Let ${}_{\cA}\Kcii = {\rm span}_{\mathcal A}\{ [A]| A \in \widetilde{\MX}_{\mm}^{\ii}\}$, and $\Kcii =\mbb Q(v) \otimes_{\cA} {}_{\cA}\Kcii$.
Similar to Proposition~\ref{prop7.1}, we can establish a stabilization property using \eqref{pii} for the algebras ${\mbf S}^{\ii}_{\mm,d}$ as $d \mapsto \infty$,
and then we can use the stabilization procedure to construct an algebra structure on $\Kcii$
(similar to Corollary~\ref{cor7.2}).  We have the following $\iiw$-counterpart of Theorem~\ref{Kji}.

\begin{prop} \label{Kii}
\begin{enumerate}
\item
The algebra $\Kcii$ is generated by $[B]$,
for $B\in \widetilde{\MX}_{\mm}^{\ii}$ such that
$B -\sum_{j \in [1,n]\backslash\{r, r+1,0,n-1\} }  c_{i,j} E^{j,j+1}_{\theta} - c_{i,0} E^{-1,1}_{\theta} - c_{i,r+1} E^{r, r+2}_{\theta}$ \redtext{is} diagonal for some
suitable scalars $c_{i,j}$.

\item
The algebra $\Kcii$ admits the monomial basis $\{ \texttt{M}_A^{\ii} \vert A\in \widetilde{\MX}_{\mm}^{\ii}\}$,
and the canonical basis $\{ \{ A\} \vert A\in \widetilde{\MX}_{\mm}^{\ii}\}$.

\item
There exists a surjective algebra homomorphism $\Psi^{\imath  \imath}: \Kcii \rightarrow \mbf S^{\imath  \imath}$
 such that $\Psi^{\imath  \imath} ([A]) = [A]$ if $A\in  {\MX}_{\mm,d}^{\imath  \imath}$ and $0$ otherwise.
 \end{enumerate}
\end{prop}

By restriction of the involution $\tau_d: \Sj \rightarrow \Sj$, we obtain involutions on the convolution algebra ${\mbf S}^{\ii}_{\mm,d}$ and Lusztig algebra ${\mbf U}^{\ii}_{\mm,d}$ in ~\cite{FLLLWa},
respectively.
Thus we have the following commutative diagram, which is a variant of Diagram \eqref{CD:autom}:
\begin{equation*}
\begin{CD}
\bU^{\ii}_{\mm,d} @>>> {\mbf S}^{\ii}_{\mm,d} @>>> \Sj @>\jmath_d >> {\mbf S}_{n,d}  \\
@V\tau_d VV @V\tau_d VV @V\tau_d VV  @V\tau^{\fa}_d VV  \\
\bU^{\ii}_{\mm,d}  @>>> {\mbf S}^{\ii}_{\mm,d}  @>>> \Sj @>\jmath_d >>{\mbf S}_{n,d}.
\end{CD}
\end{equation*}

Similar to Theorem~\ref{thm:tau} and Theorem~\ref{thm:tauK}, we can establish the following.
\begin{prop}
The involution $\tau_d$ on ${\mbf S}^{\ii}$ (or ${\mbf U}^{\ii}_{\eta, d}$)  preserves the standard, monomial and canonical bases.
Moreover, the involution $\tau_d$ induces an involution on $\Kcii$, which preserves the standard, monomial and canonical bases.
\end{prop}


\end{document}